\documentclass[sn-mathphys]{sn-jnl}

\jyear{2021}%
\usepackage{dsfont}
\usepackage{empheq}
\newtheorem{theorem}{Theorem}[section]

\newtheorem{lemma}[theorem]{Lemma}

\newtheorem{definition}[theorem]{Definition}
\newtheorem{remark}{Remark}

\def \d {\mathrm{d}}

\theoremstyle{thmstyleone}
\theoremstyle{thmstyletwo}
\theoremstyle{thmstylethree}
\begin{document}

\title[Finite-time stabilization and impulse control of heat equation]{Finite-time stabilization and impulse control of heat equation with dynamic boundary conditions}

%%=============================================================%%

\author[1]{\fnm{Salah-Eddine} \sur{Chorfi}}\email{chorphi@gmail.com}

\author[2]{\fnm{Ghita} \sur{El Guermai}}\email{ghita.el.guermai@gmail.com}
\author[3]{\fnm{Lahcen} \sur{Maniar}}\email{maniar@uca.ma}
\author*[4]{\fnm{Walid} \sur{Zouhair}}\email{walid.zouhair.fssm@gmail.com}

\affil[1]{\orgdiv{LMDP, UMMISCO (IRD-UPMC)}, \orgname{Cadi Ayyad University, Faculty of Sciences Semlalia}, \orgaddress{ \city{Marrakesh}, \postcode{B.P. 2390}, \country{Morocco}}}


\abstract{In this paper, we study the impulse controllability of a multi-dimensional heat equation with dynamic boundary conditions in a bounded smooth domain. Using a recent approach based on finite-time stabilization, we show that the system is impulse null controllable at any positive time via impulse controls supported in a nonempty open subset of the physical domain. Furthermore, we infer an explicit estimate for the exponential decay of the solution. The proof of the main result combines a logarithmic convexity estimate and some spectral properties associated to dynamic boundary conditions. In our setting, the nature of the equations, which couple intern-boundary phenomena, makes it necessary to go into quite sophisticated estimates incorporating several boundary terms.}


\maketitle

\section{Introduction}\label{sec1}
We are concerned with the impulse null controllability of heat equation subject to dynamic boundary conditions in a bounded domain. It consists of steering the solution to zero at a given positive time by stabilizing in a finite-time the system via impulse controls.

Throughout the paper, let $T>0$ be a fixed positive time and $\Omega\subset \mathbb{R}^N$ ($N \ge 2$) be a bounded smooth domain with boundary $\Gamma$ of class $C^2$. We are interested in the following impulse controlled system
\begin{empheq}[left = \empheqlbrace]{alignat=2} \label{1.1}
\begin{aligned}
&\partial_{t} \psi-\Delta \psi=0, && \qquad\text { in } \Omega \times(0, T) \setminus\displaystyle \bigcup_{k\geq 0 }\{\tau_{k}\},\\
&\psi(\cdot, \tau_{k})=\psi\left(\cdot, \tau_{k}^{-}\right)+\mathds{1}_{\omega} \mathcal{L}_{k}(\psi(\cdot,t_{k})), && \qquad\text { in } \Omega,\\
&\partial_{t}\psi_{\Gamma} - \Delta_{\Gamma} \psi_{\Gamma} + \partial_{\nu}\psi =0, && \qquad\text { on } \Gamma \times(0, T)\setminus\displaystyle \bigcup _{k\geq 0 }\{\tau_{k}\}, \\
&\psi_{\Gamma}(\cdot, \tau_{k})=\psi_{\Gamma}\left(\cdot, \tau_{k}^{-}\right), && \qquad\text { on } \Gamma,\\
& \psi_{\Gamma}(x,t) = \psi_{\mid\Gamma}(x,t), &&\qquad\text{ on } \Gamma \times(0, T) , \\
& \left(\psi(\cdot, 0),\psi_{\Gamma}(\cdot, 0)\right)=\left(\psi^{0},\psi^{0}_{\Gamma}\right), && \qquad \text{ on } \Omega\times\Gamma,
\end{aligned}
\end{empheq}
where $\left(\psi^{0},\psi^{0}_{\Gamma}\right)\in \mathbb{L}^2$ denotes the initial condition, $\displaystyle\lbrace t_{k}\rbrace_{k \ge 0}\,$ is an increasing sequence of positive real numbers and $\tau_{k} := \frac{t_{k} + t_{k+1} }{2} $, $\psi(\cdot,\tau_{k}^{-})$ denotes the left limit of the function $\psi$ at time $\tau_{k} \in (0,T)$,  $\omega\Subset \Omega$ is a nonempty open subset, and $\mathds{1}_\omega$ stands for the characteristic function of $\omega$, $\mathcal{L}_{k}$ is a linear bounded operator that we will define later. We also denote by $\psi_{\mid\Gamma}$ the trace of $\psi$, and by $\partial_{\nu} \psi:=(\nabla \psi \cdot \nu)_{\mid\Gamma}$ the normal derivative, where $\nu$ stands for the unit exterior field normal to $\Gamma$. The tangential gradient will be denoted as $\nabla_{\Gamma} \psi=\nabla \psi-\left(\partial_{\nu} \psi\right) \nu$. Let $\mathrm{g}$ denote the standard Riemannian metric on $\Gamma$ induced by $\mathbb{R}^N$. The Laplace-Beltrami operator $\Delta_\Gamma$ is defined locally as follows
\begin{equation*}
\Delta_\Gamma=\frac{1}{\sqrt{\mid\mathrm{g}\mid}} \sum_{i,j=1}^{N-1} \frac{\partial}{\partial x^i} \left(\sqrt{\mid\mathrm{g}\mid}\, \mathrm{g}^{ij} \frac{\partial}{\partial x^j}\right),\label{eqlb}
\end{equation*}
where $\mathrm{g}=\left(\mathrm{g}_{ij}\right)$ is the metric tensor, $\mathrm{g}^{-1}=\left(\mathrm{g}^{ij}\right)$ its inverse and $\mid\mathrm{g}\mid=\det\left(\mathrm{g}_{ij}\right)$. Finally, we will denote the Hessian matrix of $\psi$ with respect to $\mathrm{g}$ by $\nabla^2_\Gamma \psi$. In the sequel, we mainly use the following divergence formula
\begin{equation*}
\int_\Gamma \Delta_\Gamma u\, v \,\d S =- \int_\Gamma \langle \nabla_\Gamma u, \nabla_\Gamma v\rangle_\Gamma \,\d S, \quad u\in H^2(\Gamma), \, v\in H^1(\Gamma), \label{sdt}
\end{equation*}
where $\langle \cdot, \cdot \rangle_\Gamma$ is the Riemannian inner product of tangential vectors on $\Gamma$.

The boundary conditions of dynamic nature arise in several physical models of concrete applications, ranging from population dynamics, thermodynamics (heat transfer), fluid dynamics, etc. See for instance \citep{FH'11,La'32} and the references therein. One main feature of dynamic boundary conditions is that they derive naturally from the physical laws when one incorporates boundary conditions into the physical formulation of the problem under study. We refer to \citep{Go'06, Sa'20} for physical interpretations and mathematical derivations of several models with dynamic boundary conditions.

In the past few decades, many authors worked
on existence, stability, and controllability for impulsive dynamical systems; we mention e.g. \citep{Zouhair2022',SLAPWZ2,LHWZME2,Jose}. However, when it comes to impulsive controllability, the literature is not that broad, and not as many works are available in this area. We refer the interested reader to \citep{ABWZ,ka, MBEYR, YT}. The sharpest results are usually obtained by logarithmic convexity or spectral inequalities. The paper \citep{pkm} presents a Carleman commutator approach based on logarithmic convexity for some parabolic equations with homogeneous Dirichlet boundary conditions, while the recent paper \citep{RBKDP} investigates the observability for Neumann boundary conditions. In \citep{BP'18}, the authors have established a Lebeau–Robbiano-type spectral inequality for a degenerate one-dimensional elliptic operator with application to impulse control and finite-time stabilization for a degenerate parabolic equation. It should be pointed that this method is a new approach to steer the solution to zero by means of impulse control as a stabilizer in finite-time.

It is worth highlighting that the controllability results for non-impulsive parabolic equations have been extensively studied in the literature, see \citep{Coron2017, BAJS,CGHL18, DSMSM21} and the references therein, However, there are a few results that are available for the bulk-surface coupled systems. Recently, the controllability for non-impulsive parabolic equations with dynamic boundary conditions has been studied in \citep{BCMO'20, KM'19, MMS'17, MOZ2022}, and some related inverse problems are investigated in \citep{ACM'21, ACM'21', ACM'21'', ACM'22, ACMO'20, SEEGLMWZ22}. As for the impulse controllability for equations with dynamic boundary conditions, we have studied in \citep{CGMZ'22} the impulse approximate controllability for the multi-dimensional case. The one-dimensional system has been investigated numerically in \citep{CGMZ'221} where a constructive algorithm for computing the impulse control of minimal $L^2$-norm is presented.

In the present paper, we extend the recent strategy proposed in \citep{BP'18} for a scalar degenerate elliptic operator with the zero Dirichlet condition to a matrix operator on a product space to include the dynamics with respect to time at the boundary $\Gamma$.

Let us start with the following definition.
\begin{definition}
The system \eqref{1.1} is said to be finite-time stabilizable in time $T>0$, if for every $\left(\psi^{0},\psi^{0}_{\Gamma}\right)\in \mathbb{L}^2$ there exist control operators $\mathcal{L}_{k}$ such that the corresponding solution satisfies $\lim\limits_{t\to T^-} \Psi(t)=0$.
\end{definition}
Next, we state our main result which is a finite-time stabilization result via impulse controls for the heat equation with dynamic boundary conditions. The proof is given in Section \ref{sec4}. 
\begin{theorem}\label{thm1.4}
For $\tau_{k} = \frac{t_{k} + t_{k+1} }{2},$ with $t_{k}=T\left( 1-\dfrac{1}{b^k} \right)$ and $b>1$, the system \eqref{1.1} is finite-time stabilizable. Moreover, there exist constants $C,K > 0$ such that for any initial condition $\Psi_{0} = \left(\psi^{0},\psi^{0}_{\Gamma}\right)\in \mathbb{L}^2,$ the solution $\Psi = (\psi, \psi_{\Gamma})$ of the system \eqref{1.1} satisfies
\begin{equation*}
\|\Psi(t)\| \leq C \mathrm{e}^{-\frac{1}{K}\left(\frac{T}{T-t}\right)} \left\|\Psi_{0}\right\| \qquad \text { for any } 0 \leq t<T^{-}.
\end{equation*}
Furthermore,
$\displaystyle \lim _{k \rightarrow \infty}\left\|\mathcal{L}_{k}\left(\psi\left(\tau_{k}\right)\right)\right\|_{L^2(\omega)}=0.$
\end{theorem}
The main key to prove the above theorem is the following result of observability estimate at one point of time. The proof is given in Section \ref{sec3}.
\begin{theorem}\label{theo1.1}
Let $\Omega \subset \mathbb{R}^N$  be a bounded domain of class $C^2$ and let $\omega \Subset \Omega$ be a nonempty open subset. Let $\langle \cdot, \cdot\rangle$ denote the usual inner product in
$L^2(\Omega) \times L^2 (\Gamma)$ and let $\|\cdot \|$ be its corresponding norm. Then the following observation estimate holds
\begin{equation}\label{1.2}
\|U(\cdot, T)\| \leq \mathrm{e}^{C\left(1+\frac{1}{T}\right)}\|u(\cdot, T)\|_{L^{2}(\omega)}^{\beta}\|U(\cdot, 0)\|^{1-\beta},
\end{equation}
where $\beta \in (0,1)$, $C >0$ and $U=\left(u,u_{\Gamma}\right)$ is the solution of the following system
\begin{empheq}[left = \empheqlbrace]{alignat=2}\label{s1.3}
\begin{aligned}
&\partial_{t} u-\Delta u=0, && \qquad\text { in } \Omega \times(0, T), \\
&\partial_{t}u_{\Gamma} - \Delta_{\Gamma} u_{\Gamma} + \partial_{\nu}u =0, && \qquad\text { on } \Gamma \times(0, T), \\
& u_{\Gamma}(x,t) = u_{\mid\Gamma}(x,t), &&\qquad\text{ on } \Gamma \times(0, T) , \\
& \left(u(\cdot, 0),u_{\Gamma}(\cdot, 0)\right)=\left(u^{0},u^{0}_{\Gamma}\right), && \qquad \text{ on } \Omega\times\Gamma.
\end{aligned}
\end{empheq}
\end{theorem}
\begin{remark}
The observability inequality \eqref{1.2} generalizes the one recently obtained in \citep{CGMZ'22} for a convex domain. Here we remove this geometric restriction by using a localization argument (see Theorem \ref{theo1.2}).
\end{remark}

The proof of the observability estimate in Theorem \ref{theo1.1} is based on the following result whose proof is given in Section \ref{sec3}.
\begin{theorem}\label{theo1.2}
Let $0<r<R$ and $x_{0}\in \Omega$. Assume that $B_{x_{0},r}\subset \Omega$ and $\Omega \cap B_{x_{0},R_{0}}$ is star-shaped with respect to $x_{0}$, with $R_{0}=(1+2\delta)R$ for some $\delta\in(0,1)$. There is a nonempty open subset $\omega_{0}$ of  $B_{x_{0},r}$, constants $\beta\in (0,1)$ and $C>0$ such that for any $T>0$ and $U_{0}\in \mathbb{L}^2$,
\begin{equation}\label{1.1.2}
\|u(\cdot, T)\|_{L^2(\Omega \cap B_{x_{0},R})}+\|u_{\Gamma}(\cdot, T)\|_{L^2(\Gamma \cap B_{x_{0},R})} \leq \mathrm{e}^{C\left(1+\frac{1}{T}\right)}\|u(\cdot, T)\|_{L^{2}(\omega_{0})}^{\beta}\|U(\cdot, 0)\|^{1-\beta},
\end{equation}
where $U=\left(u,u_{\Gamma}\right)$ is the solution of \eqref{s1.3}.
\end{theorem}

The rest of the paper is organized as follows: in Section \ref{sec2}, the functional setting and some useful results associated to the heat equation with dynamic boundary conditions are presented along with the wellposedness of the corresponding systems. In Section \ref{sec3}, we elaborate the logarithmic convexity estimate for the system \eqref{s1.3}. Section \ref{sec4} is devoted to establish the finite-time stabilization via impulse controls. The impulse null controllability of the system \eqref{1.1} will be then a direct consequence.
\section{Functional setting}\label{sec2}
In this section, we recall some results that will be useful in the sequel. We will often use the following real spaces
$$\mathbb{L}^2:=L^2(\Omega, \d x)\times L^2(\Gamma, \d S) \quad \text{and} \quad \mathbb{L}_{0}^2:=L^2(\Omega\cap B_{x_{0},R_{0}}, \d x)\times L^2(\partial(\Omega\cap B_{x_{0},R_{0}}), \d S), $$
which are Hilbert spaces equipped with the inner products given by
\begin{align*}
&\langle (u,u_\Gamma),(v,v_\Gamma)\rangle_{\mathbb{L}^2}:=\langle (u,u_\Gamma),(v,v_\Gamma)\rangle =\langle u,v\rangle_{L^2(\Omega)} +\langle u_\Gamma,v_\Gamma\rangle_{L^2(\Gamma)},\\
&\langle (u,u_\Gamma),(v,v_\Gamma)\rangle_{\mathbb{L}_{0}^2}:=\langle (u,u_\Gamma),(v,v_\Gamma)\rangle_{0} =\langle u,v\rangle_{L^2(\Omega\cap B_{x_{0},R_{0}})} +\langle u_\Gamma,v_\Gamma\rangle_{L^2(\partial(\Omega\cap B_{x_{0},R_{0}}))},
\end{align*}
where the Lebesgue measure on $\Omega$ is denoted by $\d x$ and the surface measure on $\Gamma$ by $\d S$. We also consider the space
$$\mathbb{H}^k:=\left\{(u,u_\Gamma)\in H^k(\Omega)\times H^k(\Gamma)\colon u_{\mid\Gamma} =u_\Gamma \right\} \text{ for } k=1,2,$$ equipped with the standard product norm.

System \eqref{s1.3} can be written as an abstract Cauchy problem
\begin{equation}
\begin{cases}
\hspace{-0.1cm} \partial_t \mathbf{U}=\mathbf{A} \mathbf{U}, \quad 0<t \le T, \nonumber\\
\hspace{-0.1cm} \mathbf{U}(0)=(u^0, u^0_\Gamma), \nonumber
\end{cases}
\end{equation}
where $\mathbf{U}:=(u,u_{\Gamma})$. The linear operator $\mathbf{A} \colon D(\mathbf{A}) \subset \mathbb{L}^2 \longrightarrow \mathbb{L}^2$ is given by
\begin{equation*}
\mathbf{A}=\begin{pmatrix} \Delta & 0\\ -\partial_\nu & \Delta_\Gamma \end{pmatrix}, \qquad \qquad D(\mathbf{A})=\mathbb{H}^2.
\end{equation*}
The operator $\mathbf{A}$ is selfadjoint and dissipative; it generates then an analytic $C_0$-semigroup of contractions $\left(\mathrm{e}^{t\mathbf{A}}\right)_{t\geq 0}$ on $\mathbb{L}^2$ (see \citep{MMS'17}). Hence, the solution map $t\mapsto \mathrm{e}^{t\mathbf{A}} \mathbf{U}_0$ is infinitely many times differentiable for $t>0$ and $\mathrm{e}^{t\mathbf{A}} \mathbf{U}_0 \in D\left(\mathbf{A}^m\right)$ for every $\mathbf{U}_0 \in \mathbb{L}^2$ and $m\in \mathbb{N}$.

On the other hand, we rewrite system \eqref{1.1} as the impulsive Cauchy problem
\begin{equation}
\text{(ACP)} \;\; \begin{cases}
\hspace{-0.1cm} \partial_t \Psi(t)=\mathbf{A} \Psi(t), \quad (0,T)\setminus \displaystyle \bigcup _{k\geq 0 }\{\tau_{k}\}, \nonumber\\
\hspace{-0.1cm} \bigtriangleup \Psi(\tau_{k}) = (\mathds{1}_{\omega} h(t_{k}), 0),\nonumber\\
\hspace{-0.1cm} \Psi(0)=(\psi^0, \psi^0_\Gamma), \nonumber
\end{cases}
\end{equation}
where $\Psi:=(\psi,\psi_{\Gamma})$ and $\bigtriangleup \Psi(\tau_{k}) :=\Psi\left(\cdot,\tau_{k}\right) - \Psi\left(\cdot,\tau_{k}^{-}\right)$.
For all $\Psi_{0}:=(\psi^0, \psi^0_\Gamma) \in \mathbb{L}^2$, the system (ACP) has a unique mild solution given by
$$
\Psi(t) = \mathrm{e}^{t\mathbf{A}} \Psi_{0} + \sum_{k\geq 1} \mathds{1}_{\{t\geq \tau_{k} \}}(t)\, \mathrm{e}^{(t-\tau_{k})\mathbf{A}} (\mathds{1}_{\omega} h(t_{k}),0), \qquad t\in (0,T).
$$
\subsection{The weight function}
Let $\Phi \colon \overline{\Omega} \times (0,T)\rightarrow \mathbb{R} $ be the weight function defined by
\[
\Phi (x,t)=\frac{-s\left\lvert x-x_{0}\right\lvert^2}{4(T-t+h)},
\]
where $x_0 \in \omega \Subset \Omega, h > 0$ and $s \in \left(0,1\right)$. To simplify the notation, we set
\[
\varphi(x)=-\frac{\left\lvert x-x_{0}\right\lvert^2}{4}\quad \text{and}\quad \Upsilon(t)=T-t+ h,
\]
so that
$$\Phi (x,t)=\frac{s \varphi(x)}{\Upsilon(t)}.$$
The function $\varphi$ satisfies the following properties
\begin{itemize}
\item[(1)] $\varphi(x)+\left\lvert \nabla\varphi(x)\right\lvert^{2} = 0,\; \forall x\in \overline{\Omega}$,
\item[(2)] $\nabla\varphi(x)=-\dfrac{1}{2}(x-x_{0}), \;\forall x\in\overline{\Omega}$,
\item[(3)] $\Delta\varphi(x) =-\dfrac{N}{2},\; \forall x\in\overline{\Omega}$,
\item[(4)] $\nabla^2\varphi=-\dfrac{1}{2}I_{N}$ ($I_{N}$ is the identity matrix),
\item[(5)] $\partial_{\nu}\varphi(x)=-\dfrac{1}{2}(x-x_{0})\cdot \nu(x),\; \forall x\in \Gamma$.
\item[(6)] Since $\Omega \cap B_{x_{0},R_{0}}$ is star-shaped with respect to $x_{0}$, $\partial_{\nu}\varphi(x) <0,\; \forall x\in \partial(\Omega \cap B_{x_{0},R_{0}}).$
\end{itemize}
\subsection{Impulsive approximate controllability}
For $T>0$, we consider the following system with one pulse $\tau \in (0,T)$,
\begin{empheq}[left = \empheqlbrace]{alignat=2} \label{Sys5}
\begin{aligned}
&\partial_{t} \psi-\Delta \psi=0, && \qquad\text { in } \Omega \times(0, T) \backslash\{\tau\},\\
&\psi(\cdot, \tau)=\psi\left(\cdot, \tau^{-}\right)+\mathds{1}_{\omega} h(\cdot,\tau), && \qquad\text { in } \Omega,\\
&\partial_{t}\psi_{\Gamma} - \Delta_{\Gamma} \psi_{\Gamma} + \partial_{\nu}\psi =0, && \qquad\text { on } \Gamma \times(0, T)\backslash\{\tau\}, \\
&\psi_{\Gamma}(\cdot, \tau)=\psi_{\Gamma}\left(\cdot, \tau^{-}\right), && \qquad\text { on } \Gamma,\\
& \psi_{\Gamma}(x,t) = \psi_{\mid\Gamma}(x,t), &&\qquad\text{ on } \Gamma \times(0, T) , \\
& \left(\psi(\cdot, 0),\psi_{\Gamma}(\cdot, 0)\right)=\left(\psi^{0},\psi^{0}_{\Gamma}\right), && \qquad \text{ on } \Omega\times\Gamma.
\end{aligned}
\end{empheq}
\begin{definition}[see \citep{VTMN}] 
System \eqref{Sys5} is null approximate impulse controllable at time $T$ if for any $\varepsilon > 0$ and
any $\Psi^0=\left(\psi^{0},\psi^{0}_{\Gamma}\right) \in \mathbb{L}^2$, there exists a control function $h \in L^2(\omega)$ such that the associated state at final time satisfies
\begin{equation*}
\|\Psi(\cdot, T)\| \leq \varepsilon\left\Vert \Psi^0\right\Vert .
\end{equation*}
\end{definition}
\noindent This means that for every $\varepsilon >0$ and $\Psi^0=\left(\psi^{0},\psi^{0}_{\Gamma}\right) \in \mathbb{L}^2$, the set
\begin{equation*}
\mathcal{R}_{T, \Psi^{0}, \varepsilon} :=\left\{h \in L^{2}(\omega): \text { the solution of }\eqref{Sys5}\text { satisfies }\left\Vert\Psi(\cdot, T)\right\Vert \leq \varepsilon\left\Vert \Psi^{0}\right\Vert\right\},
\end{equation*}
is nonempty.

Next, we state the result on approximate impulse controllability for system \eqref{Sys5}.
\begin{lemma}[see \citep{CGMZ'22}]\label{thm1.2}
The system \eqref{Sys5} is null approximate impulse controllable at any time $T > 0$. Moreover, for any $\varepsilon > 0,$ we have
\begin{equation}
\frac{1}{\kappa^2}\|h\|_{L^{2}(\omega)}^{2}+\frac{1}{\varepsilon^2}\|\Psi(\cdot, T)\|^{2} \leq\left\|\Psi^{0}\right\|^{2},
\end{equation}
where $\kappa$ is a positive constant depending on $T$ and $\varepsilon$.
\end{lemma}
The following spectral Lemma is needed for the sequel. See \citep{Ci'16} for more details.
	\begin{lemma}\label{lmsp1}
		There exists an increasing sequence of real numbers $(\lambda_n)_{n\ge 1}$ such that 
		$$0= \lambda_1 < \lambda_2 < \cdots ,$$
		converging to $\infty$ as $n\to \infty$ such that the spectrum
		of $-\mathbf{A}$ is given by
		$$\sigma(-\mathbf{A})=\{\lambda_n \mid\; n\in \mathbb{N} \}.$$
	\end{lemma}
\section{Proof of the logarithmic convexity estimate}\label{sec3}
\subsection{Proof of Theorem \ref{theo1.2}}
For the reader convenience, we split the proof into several steps:

\textbf{Step 1. Symmetric part and antisymmetric part.} Let $\chi\in C_{0}^{\infty}(B_{x_{0},R_{0}})$ such that
\begin{equation}\label{Ass0}
 0\leq \chi\leq 1 \qquad \text{and} \qquad \chi=1 \;\; \text{on}\;\; \lbrace x;\; \mid x-x_{0}\mid\leq (1+3\delta/2)R\rbrace.
 \end{equation}
Introduce 
$$Z=\chi U,$$
such that $ Z =\begin{pmatrix}
z \\ z_{\Gamma}
\end{pmatrix}$. We get 
\begin{align*}
    &\partial_{t}z-\Delta z =g:=-2\nabla\chi\cdot \nabla u-\Delta \chi u, \;\; \text{in}\;\; \Omega \cap B_{x_{0},R_{0}},\nonumber\\
    &\partial_{t}z_{\Gamma}-\Delta_{\Gamma}z_{\Gamma}+\partial_{\nu}z=g_{\Gamma}:= -2 \langle \nabla_{\Gamma}\chi,\nabla_{\Gamma}u_{\Gamma}\rangle_{\Gamma}-\Delta_{\Gamma}\chi u_{\Gamma}\partial_{\nu}\chi u_{\Gamma},\;\; \text{on}\;\; \Gamma\cap B_{x_{0},R_{0}},\\
    &z=0, \qquad \text{on}\;\; \partial(\Omega\cap B_{x_{0},R_{0}})\left\backslash\Gamma\right..\nonumber
\end{align*}
We define
\begin{equation*}
F(x, t)=Z(x, t) \mathrm{e}^{\Phi(x, t) / 2},
\end{equation*}
such that $ F =\begin{pmatrix}
f \\ f_{1}
\end{pmatrix}$, where $f_{1}=f_{\Gamma}$ on $\Gamma\cap B_{x_{0},R_{0}}$ and $f_{1}=f=0$ on $\partial(\Omega\cap B_{x_{0},R_{0}})\left\backslash\Gamma\right.$. We introduce the operator $P$ as follows
\begin{equation*}
P F = \mathrm{e}^{\Phi / 2}
\begin{pmatrix}
\partial_{t}-\Delta & 0 \\[3mm]
\partial_{\nu} & \partial_{t}-\Delta_{\Gamma}
\end{pmatrix} \mathrm{e}^{-\Phi / 2} F=\mathrm{e}^{\Phi / 2}G,
\end{equation*}
with $G=\begin{pmatrix}
g \\ g_{1}
\end{pmatrix}$, such that $g_1=g_{\Gamma}$ on $\Gamma\cap B_{x_{0},R_0}$ and $g_1=0$ on $\partial(\Omega\cap B_{x_{0},R_{0}})\left\backslash\Gamma\right.$. Then
\begin{equation*}
P F =
\begin{pmatrix}
\partial_{t} f-\Delta f-\frac{1}{2}f\left(\partial_{t} \Phi+\frac{1}{2}\mid\nabla \Phi\mid^{2}\right)+ \nabla \Phi \cdot\nabla f +\frac{1}{2}\Delta \Phi f \\[3mm]
\partial_{t} f_{\Gamma}-\Delta_{\Gamma} f_{\Gamma}+ \partial_{\nu} f -\frac{1}{2} f_{\Gamma}\left(\partial_{t} \Phi+\frac{1}{2}\mid\nabla_{\Gamma} \Phi\mid^{2}\right)-\frac{1}{2} \partial_{\nu}\Phi f_{\Gamma} \\
+\frac{1}{2}\Delta_{\Gamma} \Phi f_{\Gamma} + \langle \nabla_{\Gamma} \Phi ,\nabla_{\Gamma} f_{\Gamma}\rangle_{\Gamma}
\end{pmatrix}.
\end{equation*}
Let us define $P_{1}$ as follows
\begin{equation*}
P_{1} F =
\begin{pmatrix}
\Delta f+\frac{1}{2}f\left(\partial_{t} \Phi+\frac{1}{2}\mid\nabla \Phi\mid^{2}\right)- \nabla \Phi\cdot \nabla f -\frac{1}{2}\Delta \Phi f \\[3mm]
\Delta_{\Gamma} f_{\Gamma}- \partial_{\nu} f +\frac{1}{2} f_{\Gamma}\left(\partial_{t} \Phi+\frac{1}{2}\mid\nabla_{\Gamma} \Phi\mid^{2}\right)+\frac{1}{2} \partial_{\nu}\Phi f_{\Gamma} \\
-\frac{1}{2}\Delta_{\Gamma} \Phi f_{\Gamma} - \langle \nabla_{
\Gamma} \Phi ,\nabla_{\Gamma} f_{\Gamma}\rangle_{\Gamma}
\end{pmatrix}.
\end{equation*}
 Hence
\begin{equation*}
\begin{pmatrix}
\partial_{t} f \\
\partial_{t} f_{\Gamma}
\end{pmatrix} = P_{1} F+\mathrm{e}^{\Phi / 2}G.
\end{equation*}
Let us compute the adjoint operator of $P_1$. For any $G = \begin{pmatrix}
g \\
g_{\Gamma}
\end{pmatrix} \in  \mathbb{H}^1$, we have  
\begin{align*}
\left\langle P_1 F , G \right\rangle_{0} &=\displaystyle\int_{\Omega\cap B_{x_{0},R_{0}}} \Delta f g \mathrm{d} x+\int_{\Omega\cap B_{x_{0},R_{0}}} \frac{1}{2} f g\left(\partial_{t} \Phi+\frac{1}{2}\mid\nabla \Phi\mid^{2}\right) \mathrm{d}x\\
&- \int_{\Omega\cap B_{x_{0},R_{0}}} g \nabla \Phi \cdot \nabla f \mathrm{d} x-\displaystyle\frac{1}{2} \int_{\Omega\cap B_{x_{0},R_{0}}}  \Delta \Phi fg \mathrm{d}x \\
&+ \int_{\Gamma\cap B_{x_{0},R_{0}}}  \Delta_{\Gamma} f_{\Gamma}g_{\Gamma} \mathrm{d} S - \int_{\Gamma\cap B_{x_{0},R_{0}}}  \partial_{\nu} f g_{\Gamma} \mathrm{d} S\\
&+\displaystyle \frac{1}{2} \int_{\Gamma\cap B_{x_{0},R_{0}}}  f_{\Gamma} g_{\Gamma}\left(\partial_{t} \Phi+\frac{1}{2}\mid\nabla_{\Gamma} \Phi\mid^{2}\right)\mathrm{d}S+\frac{1}{2} \int_{\Gamma\cap B_{x_{0},R_{0}}}  \partial_{\nu}\Phi f_{\Gamma} g_{\Gamma} \mathrm{d} S\\
 &-\displaystyle \frac{1}{2} \int_{\Gamma\cap B_{x_{0},R_{0}}} \Delta_{\Gamma} \Phi f_{\Gamma} g_{\Gamma} \mathrm{d} S - \int_{\Gamma\cap B_{x_{0},R_{0}}} g_{\Gamma}\left\langle\nabla_{\Gamma} \Phi, \nabla f_{\Gamma}\right\rangle_{\Gamma} \mathrm{d} S.
 \end{align*}
By using Green's formula, we obtain
 \begin{align*}
\left\langle P_1 F , G \right\rangle_{0}&= - \displaystyle\int_{\Gamma\cap B_{x_{0},R_{0}}}  f_{\Gamma} \partial_{\nu} g \mathrm{d} S + \int_{\Omega\cap B_{x_{0},R_{0}}}  f  \Delta g \mathrm{d} x \\
&+\displaystyle\frac{1}{2} \int_{\Omega\cap B_{x_{0},R_{0}}}  f g\left(\partial_{t} \Phi+\frac{1}{2}\mid\nabla \Phi\mid^{2}\right) \mathrm{d}x -\frac{1}{2}\int_{\Gamma\cap B_{x_{0},R_{0}}}  \partial_{\nu}\Phi f_{\Gamma} g_{\Gamma} \mathrm{d} S\\
&\displaystyle +\frac{1}{2} \int_{\Omega\cap B_{x_{0},R_{0}}} \Delta \Phi f g \mathrm{d}x + \int_{\Omega\cap B_{x_{0},R_{0}}} f \nabla \Phi \nabla g \mathrm{d}x \\
&+ \displaystyle\int_{\Gamma\cap B_{x_{0},R_{0}}} f_{\Gamma} \Delta_{\Gamma} g_{\Gamma} \mathrm{d} S
+ \frac{1}{2} \int_{\Gamma\cap B_{x_{0},R_{0}}}  f_{\Gamma} g_{\Gamma}\left(\partial_{t} \Phi+\frac{1}{2}\mid\nabla_{\Gamma} \Phi\mid^{2}\right) \mathrm{d} S \\
&+\displaystyle \frac{1}{2} \int_{\Gamma\cap B_{x_{0},R_{0}}}  \Delta_{\Gamma}\Phi f_{\Gamma} g_{\Gamma} \mathrm{d} S+\int_{\Gamma\cap B_{x_{0},R_{0}}} f_{\Gamma}\left\langle\nabla_{\Gamma}f_{\Gamma} , \nabla_{\Gamma} g_{\Gamma}\right\rangle_{\Gamma} \mathrm{d} S\\
&= \left\langle F, P_{1}^{*} G \right\rangle_{0}.
\end{align*}
Next, we introduce the following operators
\begin{equation*}
\mathcal{A} = \frac{P_1 - P_{1}^{*}}{2} = \begin{pmatrix}
-\nabla \Phi\cdot \nabla -\frac{1}{2}\Delta \Phi & 0 \\[3mm] 
0 & \frac{1}{2} \partial_{\nu}\Phi - \left\langle \nabla_{\Gamma} \Phi, \nabla_{\Gamma} \right\rangle_{\Gamma} - \frac{1}{2} \Delta_{\Gamma}\Phi
\end{pmatrix},
\end{equation*}
which is antisymmetric on $\mathbb{H}^1$,  and 
\begin{equation*}
\mathcal{S} = \frac{P_1 + P_{1}^{*}}{2} = \begin{pmatrix}
\Delta + \eta & 0 \\[3mm]
- \partial_{\nu} & \Delta_{\Gamma}+ \theta
\end{pmatrix},\\
\end{equation*}
which is symmetric on $\mathbb{H}^1$, where
\begin{align*}
\eta &= \frac{1}{2} \left( \partial_{t}\Phi + \frac{1}{2}\left\lvert\nabla \Phi \right\lvert^{2}\right), \\
\theta &= \frac{1}{2} \left( \partial_{t}\Phi + \frac{1}{2}\left\lvert\nabla_{\Gamma} \Phi \right\lvert^{2}\right).
\end{align*}
Thus
$
\partial_{t} F = \mathcal{S} F + \mathcal{A} F+\mathrm{e}^{\Phi / 2}G.
$
\smallskip

\textbf{Step 2. Energy estimates.} Multiplying the above equation by $F$, we obtain
\begin{equation*}
\frac{1}{2} \partial_{t} \|F\|_{0}^{2} - \left\langle \mathcal{S}F, F \right\rangle_{0} =\left\langle \mathrm{e}^{\Phi(x, t) / 2}G, F \right\rangle_{0}.
\end{equation*}
Next, we introduce the frequency function $\displaystyle 
\mathcal{N}=\frac{\langle-\mathcal{S} F, F\rangle_{0}}{\|F\|_{0}^{2}}$. Then
\begin{equation*}
\frac{1}{2} \partial_{t} \|F\|_{0}^{2} + \mathcal{N} \|F\|_{0}^{2}  =\left\langle \mathrm{e}^{\Phi(x, t) / 2}G, F \right\rangle_{0}.
\end{equation*}
The derivative of $\mathcal{N}$ satisfies 
\begin{equation*}
	\begin{array}[c]{ll}
		\dfrac{d}{d t} \mathcal{N} & \leq \dfrac{1}{\|F\|_{0}^{2}}\bigg( \left\langle-\left(\mathcal{S}^{\prime}+[\mathcal{S}, \mathcal{A}]\right) F, F\right\rangle_{0}-\displaystyle \int_{\Gamma} \partial_{\nu}f\left( \mathcal{A}_{1} f\right)_{\partial(\Omega\cap B_{x_{0},R_{0}})} \mathrm{d} S  \\
		&\quad+\int_{\partial(\Omega\cap B_{x_{0},R_{0}})} \partial_{\nu} f  \mathcal{A}_{2} f_{\Gamma} \mathrm{d} S + \displaystyle\int_{\Gamma\cap B_{x_{0},R_{0}}} \partial_{\nu} \Phi f_{\Gamma}  \left(\mathcal{S}_{1} f\right)_{\Gamma} \mathrm{d}S \\
		&- \int_{\Gamma\cap B_{x_{0},R_{0}}} \partial_{\nu} \Phi f_{\Gamma}  \mathcal{S}_{2} F \mathrm{d} S \bigg)+\dfrac{2}{\|F\|_{0}^2}\|\mathrm{e}^{\Phi(x, t) / 2}G\|_{0}^2, 
	\end{array}
\end{equation*}
where
$ \mathcal{S}F =\left(\begin{array}{l} \mathcal{S}_{1} f \\  \mathcal{S}_{2}\left( f, f_{\Gamma}\right)\end{array}\right)$, $ \mathcal{A}F =\left(\begin{array}{l} \mathcal{A}_{1} f \\  \mathcal{A}_{2} f_{\Gamma}\end{array}\right)$and $\left\langle [\mathcal{S}, \mathcal{A}] F, F\right\rangle_{0} = \langle\mathcal{S}  \mathcal{A} F, F\rangle_{0}-\langle\mathcal{A} \mathcal{S} F, F\rangle_{0}$.\\
Indeed,
\begin{align}\label{2.17}
\frac{\d}{\d t} \mathcal{N} &=\frac{1}{\|F\|_{0}^{4}}\left(\frac{\d}{\d t}\langle-\mathcal{S} F, F\rangle_{0}\|F\|_{0}^{2}+\langle\mathcal{S} F, F\rangle_{0} \frac{\d}{\d t}\|F\|_{0}^{2}\right)\nonumber \\
&=\frac{1}{\|F\|_{0}^{2}}\left[\left\langle-\mathcal{S}^{\prime} F, F\right\rangle_{0}-2\left\langle\mathcal{S} F, F^{\prime}\right\rangle_{0}\right]+\frac{2}{\|F\|_{0}^{4}}\langle\mathcal{S} F, F\rangle_{0}^{2}\nonumber\\
&+\frac{2}{\|F\|_{0}^{4}}\langle\mathcal{S} F, F\rangle_{0}\left\langle \mathrm{e}^{\Phi(x, t) / 2}G,F \right\rangle_{0}\nonumber \\
&=\frac{1}{\|F\|_{0}^{2}}\left[\left\langle-\mathcal{S}^{\prime} F, F\right\rangle_{0}-2\left\langle\mathcal{S} F, F^{\prime}\right\rangle_{0}\right]\nonumber\\
&+\frac{2}{\|F\|_{0}^{4}} \bigg[ \mid  \left< SF,F \right>_{0}+\dfrac{1}{2}\left< \mathrm{e}^{\Phi(x, t) / 2}G, F\right>_{0} \mid^2-\mid \dfrac{1}{2} \left< \mathrm{e}^{\Phi(x, t) / 2}G, F\right>_{0}\mid^2\bigg]\nonumber\\
&\leq\frac{1}{\|F\|_{0}^{2}}\left[\left\langle-\mathcal{S}^{\prime} F, F\right\rangle_{0}-2\langle\mathcal{S} F, \mathcal{A} F\rangle_{0}\right]-\frac{2}{\|F\|_{0}^{2}}\left[-\|\mathcal{S} F\|_{0}^{2}+\langle\mathcal{S} F, \mathrm{e}^{\Phi(x, t) / 2}G\rangle_{0}\right]\nonumber\\
&+\frac{2}{\|F\|_{0}^{4}}\left\|SF+\frac{1}{2}\mathrm{e}^{\Phi(x, t) / 2}G\right\|_{0}^2\|F\|_{0}^2\nonumber\\
&\leq \frac{1}{\|F\|_{0}^{2}}\left[\left\langle-\mathcal{S}^{\prime} F, F\right\rangle_{0}-2\langle\mathcal{S} F, \mathcal{A} F\rangle_{0}\right]+\frac{2}{\|F\|_{0}^{2}}\left \|\frac{1}{2}\mathrm{e}^{\Phi(x, t) / 2}G\right \|_{0}^2.
\end{align}
On the other hand, we have
\begin{align*}
\langle\mathcal{S} F, \mathcal{A} F\rangle &= \displaystyle\int_{\Omega\cap B_{x_{0},R_{0}}} \left( \Delta f + \eta f\right) \mathcal{A}_{1}f \mathrm{d} x + \int_{\partial(\Omega\cap B_{x_{0},R_{0}})} \left( - \partial_{\nu} f + \Delta_{\Gamma} f_{1}+\theta f_{1}\right) \mathcal{A}_{2}f_{1}\mathrm{d}S\\
&=\displaystyle \int_{\partial(\Omega\cap B_{x_{0},R_{0}})} \partial_{\nu} f \left(\mathcal{A}_{1}f\right)_{\partial(\Omega\cap B_{x_{0},R_{0}})} \mathrm{d} S - \int_{\Omega\cap B_{x_{0},R_{0}}} \nabla f \cdot\nabla\left(\mathcal{A}_{1}f\right) \mathrm{d}x\\
&+ \int_{\Omega\cap B_{x_{0},R_{0}}}\eta f \mathcal{A}_{1}f \d x-\displaystyle\int_{\Gamma\cap B_{x_{0},R_{0}}} \partial_{\nu} f \mathcal{A}_{2} f_{\Gamma} \mathrm{d} S  \\
&-\displaystyle\int_{\Gamma\cap B_{x_{0},R_{0}}} \left\langle\nabla_{\Gamma}f_{\Gamma} , \nabla_{\Gamma}\mathcal{A}_{2} f_{\Gamma}\right\rangle_{\Gamma} \mathrm{d} S + \displaystyle\int_{\Gamma\cap B_{x_{0},R_{0}}}\theta f_{\Gamma} \mathcal{A}_{2}f_{\Gamma} \mathrm{d} S\\
&=\displaystyle\int_{\partial(\Omega\cap B_{x_{0},R_{0}})} \partial_{\nu} f \left(\mathcal{A}_{1}f\right)_{\partial(\Omega\cap B_{x_{0},R_{0}})} \mathrm{d} S  - \int_{\Gamma\cap B_{x_{0},R_{0}}} f_{\Gamma} \partial_{\nu}\left(\mathcal{A}_{1}f\right) \mathrm{d}S\\
&+\displaystyle\int_{\Omega\cap B_{x_{0},R_{0}}} f \Delta\left(\mathcal{A}_{1}f\right) \mathrm{d}x + \int_{\Omega\cap B_{x_{0},R_{0}}}\eta f \mathcal{A}_{1}f \mathrm{d}x-\int_{\Gamma\cap B_{x_{0},R_{0}}} \partial_{\nu} f \mathcal{A}_{2} f_{\Gamma} \mathrm{d} S \\
&\displaystyle + \int_{\Gamma\cap B_{x_{0},R_{0}}} f_{\Gamma} \Delta_{\Gamma}\left(\mathcal{A}_{2} f_{\Gamma}\right) \mathrm{d} S + \int_{\Gamma\cap B_{x_{0},R_{0}}}\theta f_{\Gamma} \mathcal{A}_{2}f_{\Gamma} \mathrm{d} S\\
&= \displaystyle\int_{\Omega\cap B_{x_{0},R_{0}}}\left( \Delta + \eta \right)  \left(\mathcal{A}_{1}f\right) f \mathrm{d}x+\int_{\partial(\Omega\cap B_{x_{0},R_{0}})} \partial_{\nu} f \left(\mathcal{A}_{1}f\right)_{\partial(\Omega\cap B_{x_{0},R_{0}})} \mathrm{d} S \\
&+ \int_{\Gamma\cap B_{x_{0},R_{0}}} \left( -\partial_{\nu} \left( \mathcal{A}_{1} f \right)+ \left( \Delta_{\Gamma}+\theta\right)\left( \mathcal{A}_{2} f_{\Gamma} \right)\right) f_{\Gamma} \mathrm{d} S  \\
&-\int_{\Gamma\cap B_{x_{0},R_{0}}} \partial_{\nu} f \mathcal{A}_{2} f_{\Gamma} \mathrm{d} S.
\end{align*}
Therefore, we obtain
\begin{align}\label{3.15}
\langle\mathcal{S} F, \mathcal{A} F\rangle_{0}& = \langle\mathcal{S}  \mathcal{A} F, F\rangle_{0}+\int_{\partial(\Omega\cap B_{x_{0},R_{0}})} \partial_{\nu} f \left(\mathcal{A}_{1}f\right)_{\partial(\Omega\cap B_{x_{0},R_{0}})} \mathrm{d} S\nonumber\\
&-\int_{\Gamma\cap B_{x_{0},R_{0}}} \partial_{\nu} f \mathcal{A}_{2} f_{\Gamma} \mathrm{d} S.
\end{align}
Similarly, we prove that 
\begin{equation}\label{3.16}
\langle\mathcal{S} F, \mathcal{A} F\rangle_{0} = - \langle\mathcal{A} \mathcal{S}   F, F\rangle_{0} - \int_{\Gamma\cap B_{x_{0},R_{0}}} \partial_{\nu}\Phi f_{\Gamma} \left( \mathcal{S}_{1} f \right)_{\Gamma}\mathrm{d} S + \int_{\Gamma\cap B_{x_{0},R_{0}}}  \partial_{\nu}\Phi f_{\Gamma}  \mathcal{S}_{2}F \mathrm{d} S.
\end{equation}
By combining \eqref{3.15} and \eqref{3.16}, we obtain
\begin{align}\label{2.20}
 2 \langle\mathcal{S} F, \mathcal{A} F\rangle_{0} =& \langle\mathcal{S}  \mathcal{A} F, F\rangle_{0}-\langle\mathcal{A}  \mathcal{S} F, F\rangle_{0} + \int_{\partial(\Omega\cap B_{x_{0},R_{0}})} \partial_{\nu} f \left(\mathcal{A}_{1}f\right)_{\partial(\Omega\cap B_{x_{0},R_{0}})} \mathrm{d} S \nonumber\\
 &- \int_{\Gamma\cap B_{x_{0},R_{0}}} \partial_{\nu}f  \mathcal{A}_{2} f_{\Gamma}\mathrm{d} S- \int_{\Gamma\cap B_{x_{0},R_{0}}} \partial_{\nu}\Phi f_{\Gamma} \left( \mathcal{S}_{1} f \right)_{\Gamma}\mathrm{d} S \nonumber\\
 &+ \int_{\Gamma\cap B_{x_{0},R_{0}}}  \partial_{\nu}\Phi f_{\Gamma}  \mathcal{S}_{2}F \mathrm{d} S.
 \end{align}
Combining the equalities \eqref{2.17},\eqref{2.20} yields the desired formula. 
\bigskip

\textbf{Step 3.\, Calculating the Carleman commutator.} 
The proof of the following inequality \eqref{9.1} (respectively \eqref{14.11}) is exactly the same as the proof presented in Step 3 (respectively step 4) of the paper \citep{CGMZ'22}.
\begin{align}\label{9.1}
& \left \langle -(\mathcal{S}^\prime+[\mathcal{S},A])F,F\right \rangle_{0}-\int_{\partial(\Omega\cap B_{x_{0},R_{0}})} \partial_{\nu} f \left(\mathcal{A}_{1}f\right)_{\partial(\Omega\cap B_{x_{0},R_{0}})} \mathrm{d} S\nonumber\\
&+\int_{\Gamma\cap B_{x_{0},R_{0}}}\partial_{\nu}f A_{2}f_{\Gamma}\mathrm{d}S+\int_{\Gamma\cap B_{x_{0},R_{0}}}\partial_{\nu}\Phi f_{\Gamma}(\mathcal{S}_{1}f)_{\Gamma}\mathrm{d} S-\int_{\Gamma\cap B_{x_{0},R_{0}}}\partial_{\nu}\Phi f_{\Gamma}\mathcal{S}_{2}F\mathrm{d}S\nonumber\\
&=\dfrac{-s}{\Upsilon^3}\int_{\Omega\cap B_{x_{0},R_{0}}}\left( \varphi+\dfrac{s}{2}\mid\nabla \varphi\mid^2\right) \mid f\mid^2\mathrm{d}x+\dfrac{s}{\Upsilon}\int_{\Omega\cap B_{x_{0},R_{0}}}\mid\nabla f\mid^2\mathrm{d}x\nonumber\\
&+\dfrac{s}{\Upsilon}\int_{\partial(\Omega\cap B_{x_{0},R_{0}})}\partial_{\nu}\varphi \mid\partial_{\nu} f\mid^2\mathrm{d}x-\dfrac{s^2(2-s)}{4\Upsilon^3}\int_{\Omega\cap B_{x_{0},R_{0}}}\mid\nabla\varphi\mid^2 \mid f\mid^2\mathrm{d}x\nonumber\\
&-\dfrac{s}{\Upsilon^3}\int_{\Gamma\cap B_{x_{0},R_{0}}}\left( \varphi+\dfrac{s}{2}\mid\nabla_{\Gamma} \varphi\mid^2\right) \mid f_{\Gamma}\mid^2\mathrm{d}S-\dfrac{2s}{\Upsilon}\int_{\Gamma\cap B_{x_{0},R_{0}}}\nabla_{\Gamma}^2\varphi(\nabla_{\Gamma} f_{\Gamma},\nabla_{\Gamma} f_{\Gamma})\mathrm{d}S\nonumber\\
&+\dfrac{s}{\Upsilon}\int_{\Gamma\cap B_{x_{0},R_{0}}}(\Delta \varphi+\partial_{\nu}\varphi-\Delta_{\Gamma}\varphi)\partial_{\nu}f f_{\Gamma}\mathrm{d}S+\dfrac{s}{2\Upsilon}\int_{\Gamma\cap B_{x_{0},R_{0}}}(\Delta_{\Gamma}^2\varphi-\Delta_{\Gamma}\partial_{\nu}\varphi)\mid f_{\Gamma}\mid^2\mathrm{d}S\nonumber\\
& -\dfrac{s^2}{2\Upsilon^3}\int_{\Gamma\cap B_{x_{0},R_{0}}}\left(\mid\nabla_{\Gamma}\varphi\mid^2+s\nabla_{\Gamma}^2\varphi(\nabla_{\Gamma}\varphi,\nabla_{\Gamma}\varphi)\right) \mid f_{\Gamma}\mid^2\mathrm{d}S\nonumber\\
&+\dfrac{s^3}{4\Upsilon^3}\int_{\Gamma\cap B_{x_{0},R_{0}}}(\partial_{\nu}\varphi)^3\mid f_{\Gamma}\mid^2\mathrm{d}S.
\end{align}
%\noindent\textbf{Step 4.}
Next we estimate each term of this formula which requires making some positive terms small. This can be established thanks to the parameter $s\in (0,1)$ which is taken sufficiently small. Therefore, For any $\hbar\in(0,1]$ we prove that
\begin{align}\label{14.11}
&\left\langle-(\mathcal{S}^\prime+[\mathcal{S},A])F,F\right\rangle_{0}-\int_{\partial(\Omega\cap B_{x_{0},R_{0}})}\partial_{\nu}f (A_{1}f)_{\mid\partial(\Omega\cap B_{x_{0},R_{0}})}\mathrm{d}S \nonumber\\
&+\int_{\Gamma\cap B_{x_{0},R_{0}}}\partial_{\nu}f A_{2}f_{\Gamma}\mathrm{d}S+\int_{\Gamma\cap B_{x_{0},R_{0}}}\partial_{\nu}\Phi f_{\Gamma}(\mathcal{S}_{1}f)_{\Gamma}\mathrm{d} S-\int_{\Gamma\cap B_{x_{0},R_{0}}}\partial_{\nu}\Phi f_{\Gamma}\mathcal{S}_{2}F\mathrm{d}S\nonumber\\
&\leq \dfrac{1+C_{0}}{\Upsilon}\left\langle-\mathcal{S}F,F\right\rangle+\frac{C}{\hbar^2}\|F\|^2,
\end{align}
where $C=C(\overline{\Omega})> 0$ and $C_{0}=C(s)\in (0,1)$.
\smallskip

\textbf{Step 4. Intermediate estimates.} The following differential inequalities hold
\begin{empheq}[left = \empheqlbrace]{alignat=2}
\begin{aligned}\label{dIN}
&\mid\frac{1}{2}\frac{\d}{\d t}\left\Vert F\left(  \cdot,t\right)
\right\Vert ^{2}+\mathcal{N}\left(t\right)  \left\Vert F\left(
\cdot,t\right)  \right\Vert ^{2}\mid\leq \left\| \mathrm{e}^{\Phi /2}G(\cdot,t) \right\|_{0}\| F(\cdot,t) \|_{0}, \\
&\frac{\d}{\d t}\mathcal{N}\left(  t\right)  \leq\frac
{1+C_{0}}{\Upsilon\left(t\right)  }\mathcal{N}\left(  t\right)+\dfrac{\left\| \mathrm{e}^{\Phi(\cdot,t)/2}G(\cdot,t) \right\|_{0}^2}{\|F(\cdot,t)\|_{0}^2}+\frac{C}{h^2}.
\end{aligned}
\end{empheq}
Then, for any $0<t_{1}<t_{2}<t_{3}\leq T$, we obtain
\begin{align}\label{Inn0}
\left(\left\Vert F\left(\cdot,t_{2}\right)  \right\Vert_{0} ^{2}\right)
^{1+M}\leq\left(  \left\Vert F\left(  \cdot,t_{1}\right)  \right\Vert_{0}
^{2}\right)  ^{M}\left\Vert F\left(  \cdot,t_{3}\right)  \right\Vert_{0} ^{2}\mathrm{e}^{D},%
\end{align}
where%
\[
M=\dfrac{\displaystyle\int_{t_{2}}^{t_{3}}\dfrac{1}{(T-t+h)^{1+C_{0}}} \mathrm{d}t }{\displaystyle\int_{t_{1}}^{t_{2}}\dfrac{1}{(T-t+h)^{1+C_{0}}} \mathrm{d}t},\] 
and
\begin{align*}
    D&=M(t_{2}-t_{1})^2 \dfrac{C}{h^2}+2M(t_{2}-t_{1})\int_{t_{1}}^{t_{2}}\dfrac{\left\Vert  \mathrm{e}^{\Phi(\cdot,t) /2}G(\cdot,t)\right\Vert_{0}^2}{\|F(\cdot,t)\|_{0}^2}\mathrm{d}t\\
    &+\int_{t_{2}}^{t_{3}}\left(\dfrac{T-t_{2}+h}{T-t+h}\right)^{1+C_{0}}\mathrm{d}t\int_{t_{2}}^{t_{3}}\left( \dfrac{C}{h^2}+\dfrac{\left\Vert\mathrm{e}^{\Phi(\cdot,t)}G(\cdot,t)\right\Vert_{0}^2}{\|F(\cdot,t)\|_{0}^2} \right)\mathrm{d}t\\
    &+M\int_{t_{1}}^{t_{2}}\dfrac{\left\Vert  \mathrm{e}^{\Phi(\cdot,t) /2}G(\cdot,t)\right\Vert_{0}}{\|F(\cdot,t)\|_{0}}\mathrm{d}t+\int_{t_{2}}^{t_{3}}\dfrac{\left\Vert\mathrm{e}^{\Phi(\cdot,t)}G(\cdot,t)\right\Vert_{0}}{\|F(\cdot,t)\|_{0}}\mathrm{d}t.
\end{align*}
Indeed,
\begin{align}\label{eq16}
    \left( (T-t+h)^{1+C_{0}}N(t) \right)^{\prime}&=(T-t+h)^{1+C_{0}}\dfrac{\mathrm{d}}{\mathrm{d}t}\mathcal{N}(t)-(1+C_{0})(T-t+h)^{C_{0}}\mathcal{N}(t) \nonumber\\
    &=(T-t+h)^{1+C_{0}} \bigg[ \dfrac{\mathrm{d}}{\mathrm{d}t}\mathcal{N}(t)-\dfrac{1+C_{0}}{T-t+h} \mathcal{N}(t)\bigg]\nonumber\\
    &\leq (T-t+h)^{1+C_{0}} \bigg[ \dfrac{C}{h^2}+\dfrac{\left\Vert\mathrm{e}^{\Phi(\cdot,t)/2} G(\cdot,t)\right\Vert_{0}^2}{\|F(\cdot,t)\|_{0}^2}\bigg].
\end{align}
We integrate \eqref{eq16} over $(t,t_{2})$, we obtain
\begin{align*}
    &(T-t_{2}+h)^{1+C_{0}}\mathcal{N}(t_{2})-(T-t+h)^{1+C_{0}}\mathcal{N}(t)\\
    &\leq \int_{t}^{t_{2}}(T-\tau+h)^{1+C_{0}}\bigg( \dfrac{C}{h^2}+\dfrac{\|\mathrm{e}^{\Phi/2}G(\cdot,\tau)\|_{0}^2}{\|F(\cdot,\tau)\|_{0}^2} \bigg)\mathrm{d}\tau.
\end{align*}
Therefore,
\begin{align*}
    &\bigg( \dfrac{T-t_{2}+h}{T-t+h} \bigg)^{1+C_{0}}\mathcal{N}(t_{2})-\mathcal{N}(t)\\
    &\leq \int_{t}^{t_{2}}\bigg( \dfrac{T-\tau+h}{T-t+h} \bigg)^{1+C_{0}}\bigg( \dfrac{C}{h^2}+\dfrac{\|\mathrm{e}^{\Phi(\cdot,\tau)/2}G(\cdot,\tau)\|_{0}^2}{\|F(\cdot,\tau)\|_{0}^2} \bigg)\mathrm{d}\tau.
\end{align*}
Since $t_{1}\leq t \leq \tau \leq t_{2}$, then $\dfrac{T-\tau+h}{T-t+h}\leq 1$, Hence
\begin{align*}
    \bigg(\dfrac{T-t_{2}+h}{T-t+h}\bigg)^{1+C_{0}}\mathcal{N}(t_{2})-\int_{t_{1}}^{t_{2}}\bigg( \dfrac{C}{h^2}+\dfrac{\|\mathrm{e}^{\Phi(\cdot,t)/2}G(\cdot,t\|_{0}^2}{\|F(\cdot,t)\|_{0}^2} \bigg)\mathrm{d}t\leq \mathcal{N}(t).
\end{align*}
Using the first inequality in \eqref{dIN}, we obtain 
\begin{align*}
    &\dfrac{1}{2}\dfrac{\mathrm{d}}{\mathrm{d}t}\|F\|_{0}^2+\bigg[ \dfrac{(T-t_{2}+h)^{1+C_{0}}}{(T-t+h)^{1+C_{0}}}\mathcal{N}(t_{2})-\int_{t_{1}}^{t_{2}}\bigg( \dfrac{C}{h^2}+\dfrac{\|\mathrm{e}^{\Phi(\cdot,t)/2}G(\cdot,t)\|_{0}^2}{\|F(\cdot,t)\|_{0}^2} \bigg)\mathrm{d}t\bigg] \|F\|_{0}^2\\
    & \leq\dfrac{1}{2}\dfrac{\mathrm{d}}{\mathrm{d}t}\|F\|_{0}^2+N(t)\|F\|_{0}^2\\
    &\leq \|\mathrm{e}^{\Phi/2}G(\cdot,t)\|_{0}\|F(\cdot,t)\|_{0}.
\end{align*}
Consequently, we obtain
\begin{align*}
    &\dfrac{1}{2}\dfrac{\mathrm{d}}{\mathrm{d}t}\|F(\cdot,t)\|_{0}^2+\bigg[ \dfrac{(T-t_{2}+h)^{1+C_{0}}}{(T-t+h)^{1+C_{0}}}\mathcal{N}(t_{2})-\int_{t_{1}}^{t_{2}}\bigg( \dfrac{C}{h^2}+\dfrac{\|\mathrm{e}^{\Phi(\cdot,t)/2}G(\cdot,t)\|_{0}^2}{\|F(\cdot,t)\|_{0}^2} \bigg)\mathrm{d}t\\
    &-\dfrac{\|\mathrm{e}^{\Phi/2}G(\cdot,t)\|_{0}}{\|F(\cdot,t)\|_{0}}\bigg] \|F(\cdot,t)\|_{0}^2\leq 0.
\end{align*}
To simplify, we note 
\[
\alpha(t)=\bigg[ \dfrac{(T-t_{2}+h)^{1+C_{0}}}{(T-t+h)^{1+C_{0}}}\mathcal{N}(t_{2})-\int_{t_{1}}^{t_{2}}\bigg( \dfrac{C}{h^2}+\dfrac{\|\mathrm{e}^{\Phi(\cdot,t)/2}G(\cdot,t)\|_{0}^2}{\|F(\cdot,t)\|_{0}^2} \bigg)\mathrm{d}t-\dfrac{\|\mathrm{e}^{\Phi/2}G(\cdot,t)\|_{0}}{\|F(\cdot,t)\|_{0}}\bigg].
\]
 Then, we solve
\[
\dfrac{\mathrm{d}}{\mathrm{d}t}\|F(\cdot,t)\|_{0}^2+2\alpha(t)\|F\|_{0}^2\leq 0.
\]
Since $F(\cdot,t)$ does not vanish, we have 
\begin{align*}
    \dfrac{\mathrm{d}}{\mathrm{d}t}\|F(\cdot,t)\|_{0}^2+2\alpha(t)\|F\|_{0}^2\leq 0 &\Longleftrightarrow \dfrac{\dfrac{\mathrm{d}}{\mathrm{d}t}\|F(\cdot,t)\|_{0}^2}{\|F\|_{0}^2}\leq -2\alpha(t)\\
    &\Longleftrightarrow \dfrac{\mathrm{d}}{\mathrm{d}t}\left( \ln\left( \|F(\cdot,t)\|_{0}^2 \right) \right)\leq -2\alpha(t).
\end{align*}
We integrate the last inequality over $(t_{1},t_{2})$, we obtain
\[
\ln \|F(\cdot,t_{2})\|_{0}^2-\ln \|F(\cdot,t_{1})\|_{0}^2\leq -2 \int_{t_{1}}^{t_{2}}\alpha(t)\mathrm{d}t.
\]
Thus,
\[
\dfrac{\|F(\cdot,t_{2})\|_{0}^2}{\|F(\cdot,t_{1})\|_{0}^2}\leq \mathrm{e}^{-2\int_{t_{1}}^{t_{2}}\alpha(t)\mathrm{d}t}.
\]
Using the expression of $\alpha(t)$, we obtain
\begin{align}\label{Inn1}
    &\mathrm{e}^{2\mathcal{N}(t_{2})\displaystyle\int_{t_{1}}^{t_{2}} \bigg( \dfrac{T-t_{2}+h}{T-t+h} \bigg)^{1+C_{0}}\mathrm{d}t}\nonumber\\
    &\leq \dfrac{\|F(\cdot,t_{1})\|_{0}^2}{\|F(\cdot,t_{2})\|_{0}^2}\mathrm{e}^{ 2(t_{2}-t_{1})^2\dfrac{C}{h^2}+2(t_{2}-t_{1})\displaystyle\int_{t_{1}}^{t_{2}}\dfrac{\|\mathrm{e}^{\Phi(\cdot,t)/2}G(\cdot,t)\|_{0}^2}{\|F(\cdot,t)\|_{0}^2}\mathrm{d}t}\nonumber\\
    &\qquad\qquad\qquad\times\mathrm{e}^{2\displaystyle\int_{t_{1}}^{t_{2}}\dfrac{\|\mathrm{e}^{\Phi(\cdot,t)/2}G(\cdot,t)\|_{0}}{\|F(\cdot,t)\|_{0}} \mathrm{d}t}
\end{align}
Similarly, we prove that 
\begin{align}\label{Inn2}
    \|F(\cdot,t_{2})\|_{0}^2&\leq \|F(\cdot,t_{3})\|_{0}^2 \mathrm{e}^{2 \mathcal{N}(t_{2})\displaystyle\int_{t_{2}}^{t_{3}}\left(  \dfrac{t-t_{2}+h}{T-t+h}\right)^{1+C_{0}}\mathrm{d}t }\nonumber\\
    & \mathrm{e}^ { 2\displaystyle\int_{t_{2}}^{t_{3}}\left( \dfrac{T-t_{2}+h}{T-t+h} \right)^{1+C_{0}}\mathrm{d}t\int_{t_{2}}^{t_{3}}\left( \dfrac{C}{h^2}+\dfrac{\|\mathrm{e}^{\Phi(\cdot,t)/2}G(\cdot,t)\|_{0}^2}{\|F(\cdot,t)\|_{0}^2} \right)\mathrm{d}t } \nonumber\\
    & \mathrm{e}^{ 2\displaystyle\int_{t_{2}}^{t_{3}} \dfrac{\|\mathrm{e}^{\Phi(\cdot,t)/2}G(\cdot,t)\|_{0}}{\|F(\cdot,t)\|_{0}}\mathrm{d}t }.
\end{align}
By \eqref{Inn1} and \eqref{Inn2}, we obtain
\begin{align*}
    \|F(\cdot,t_{2})\|_{0}^2&\leq \|F(\cdot,t_{3})\|_{0}^2 \bigg( \dfrac{\|F(\cdot,t_{1})\|_{0}^2}{\|F(\cdot,t_{2})\|_{0}^2} \bigg)^M \mathrm{e}^{2M \displaystyle\int_{t_{1}}^{t_{2}}\dfrac{\|\mathrm{e}^{\Phi(\cdot,t)/2}G(\cdot,t)\|_{0}}{\|F(\cdot,t)\|_{0}} \mathrm{d}t}\\
    &\times\mathrm{e}^{ 2M(t_{2}-t_{1})^2 \dfrac{C}{h^2}+2M(t_{2}-t_{1})\displaystyle\int_{t_{1}}^{t_{2}}\dfrac{\|\mathrm{e}^{\Phi(\cdot,t)/2}G(\cdot,t)\|_{0}^2}{\|F(\cdot,t)\|_{0}^2}\mathrm{d}t }\\
    &\times\mathrm{e}^ { 2\displaystyle\int_{t_{2}}^{t_{3}}\left( \dfrac{T-t_{2}+h}{T-t+h} \right)^{1+C_{0}}\mathrm{d}t\int_{t_{2}}^{t_{3}}\left( \dfrac{C}{h^2}+\dfrac{\|\mathrm{e}^{\Phi(\cdot,t)/2}G(\cdot,t)\|_{0}^2}{\|F(\cdot,t)\|_{0}^2} \right)\mathrm{d}t } \\
    &\times \mathrm{e}^{ 2\displaystyle\int_{t_{2}}^{t_{3}} \dfrac{\|\mathrm{e}^{\Phi(\cdot,t)/2}G(\cdot,t)\|_{0}}{\|F(\cdot,t)\|_{0}}\mathrm{d}t }.
\end{align*}
This leads to the inequality \eqref{Inn0}.
\smallskip

\textbf{Step 5.} \textbf{Estimating the following term $\dfrac{\|\mathrm{e}^{\Phi(\cdot,t)}G\|_{0}^2}{\| F \|_{0}^2}$}.
The following Lemma is the main key to estimate our desired term.
\begin{lemma}\label{Lm2.1}
For any $T-\theta\leq t\leq T$ one has 
\begin{equation}
    \dfrac{\|U_{0}\|^2}{\|u(\cdot,t)\|_{L^2(\Omega\cap B_{x_{0},(1+\delta)R})}^2+\|u_{\Gamma}(\cdot,t)\|_{L^2(\Gamma\cap B_{x_{0},(1+\delta)R)}}^2}\leq \mathrm{e}^{(1+\delta)\delta \frac{R^2}{2\theta}},
\end{equation}
where 
\[
\dfrac{1}{\theta}=\dfrac{2}{(\delta R)^2}\ln \left( 2 \mathrm{e}^{R^2(1+\dfrac{1}{T})}\dfrac{\|U_{0}\|^2}{\|u(\cdot,T)\|_{L^2(\Omega \cap B_{x_{0},R})}^2+\|u_{\Gamma}(\cdot,T)\|_{L^2(\Gamma\cap B_{x_{0},R})}^2} \right),
\]
with $0< \theta \leq \min \left( 1,\dfrac{T}{2} \right)$.
\end{lemma}
\begin{proof}
Since $U$ is solution of \eqref{s1.3}, then $U(x,t)=\mathrm{e}^{t\mathcal{A}}U_{0}$ with $U_{0}\in \mathbb{L}^2$ non-null initial data. Recall that for any locally lipschitz function $\zeta (x,t)$ such that $\partial_{t}\zeta +\dfrac{1}{2}\mid\nabla \zeta\mid^2\leq 0$, the function $t\longmapsto I(t)=\|U \mathrm{e}^{\zeta/2}\|^2$ is a decreasing function in t. We can choose for example $\zeta(x,t)=\dfrac{-\mid x-x_{0}\mid^2}{2(T-t+h)}$.

Now, we calculate the derivative of the function $I$, we have
\begin{align*}
    I^{\prime}(t)&=\dfrac{\mathrm{d}}{\mathrm{d}t}\bigg(  \int_{\Omega} \mid u\mid^2 \mathrm{e}^{\zeta}\mathrm{d}x+ \int_{\Gamma} \mid u_{\Gamma}\mid^2 \mathrm{e}^{\zeta}\mathrm{d}S  \bigg)\\
    &=\int_{\Omega} \mid u\mid^2 \zeta_{t} \mathrm{e}^{\zeta}\mathrm{d}x+2\int_{\Omega} u u_{t}\mathrm{e}^{\zeta} \mathrm{d}x+\int_{\Gamma}\mid u_{\Gamma}\mid^2 \zeta_{t} \mathrm{e}^{\zeta}\mathrm{d}S\\
    &+2\int_{\Gamma} u_{\Gamma} u_{\Gamma t}\mathrm{e}^{\zeta} \mathrm{d}S.
\end{align*}
Using \eqref{s1.3}$_{1}$ and \eqref{s1.3}$_{2}$, we obtain
\begin{align*}
    I^{\prime}(t)&=\int_{\Omega} \mid u\mid^2 \zeta_{t} \mathrm{e}^{\zeta}\mathrm{d}x+2\int_{\Omega} u \Delta u\mathrm{e}^{\zeta} \mathrm{d}x+\int_{\Gamma}\mid u_{\Gamma}\mid^2 \zeta_{t} \mathrm{e}^{\zeta}\mathrm{d}S\\
    &+2\int_{\Gamma} u_{\Gamma} \Delta_{\Gamma}u_{\Gamma }\mathrm{e}^{\zeta} \mathrm{d}S-2\int_{\Gamma} u_{\Gamma} \partial_{\nu}u\mathrm{e}^{\zeta} \mathrm{d}S\\
    &=\int_{\Omega} \mid u\mid^2 \zeta_{t} \mathrm{e}^{\zeta}\mathrm{d}x+2\int_{\Gamma} u\partial_{\nu}u \mathrm{e}^{\zeta}\mathrm{d}S-\int_{\Omega }\nabla (2u\mathrm{e}^{\zeta})\cdot \nabla u \mathrm{d}x\\
    &+\int_{\Gamma} \mid u_{\Gamma}\mid^2\zeta_{t}\mathrm{e}^{\zeta}\mathrm{d}S-\int_{\Gamma}\langle \nabla_{\Gamma}(2u_{\Gamma}\mathrm{e}^{\zeta}),\nabla_{\Gamma}u_{\Gamma} \rangle_{\Gamma}\mathrm{d}S-\int_{\Gamma}2u_{\Gamma}\mathrm{e}^{\zeta}\partial_{\nu}u \mathrm{d}S\\
    &=\int_{\Omega}\mid u\mid^2 \zeta_{t} \mathrm{e}^{\zeta}\mathrm{d}x - 2\int_{\Omega} \nabla \zeta\cdot \nabla u u \mathrm{e}^{\zeta} \mathrm{d}x -2 \int_{\Omega}\mid\nabla u\mid^2 \mathrm{e}^{\zeta} \mathrm{d}x\\
    &+\int_{\Gamma}\zeta_{t}\mid u_{\Gamma}(x,t)\mid^2\mathrm{e}^{\zeta}\mathrm{d}S \mathrm{d}t-2\int_{\Gamma}\langle \nabla_{\Gamma}\zeta, \nabla_{\Gamma}u_{\Gamma} \rangle_{\Gamma}u_{\Gamma}\mathrm{e}^{\zeta}\mathrm{d}S\\
    &-2\int_{\Gamma}\mid\nabla_{\Gamma}u_{\Gamma}\mid^2 \mathrm{e}^{\zeta}\mathrm{d}S.
\end{align*}
Using the fact that $\partial_{t}\zeta+\dfrac{1}{2}\mid\nabla\zeta\mid^2\leq 0$, we obtain
\begin{align*}
    I^{\prime}(t)&\leq \dfrac{-1}{2}\int_{\Omega}\mid\nabla\zeta\mid^2\mid u\mid^2\mathrm{e}^{\zeta}\mathrm{d}x-2\int_{\Omega}\nabla \zeta\cdot\nabla u u \mathrm{e}^{\zeta}\mathrm{d}x-2 \int_{\Omega}\mid\nabla u\mid^2\mathrm{e}^{\zeta} \mathrm{d}x\\
    &\dfrac{-1}{2}\int_{\Gamma}\mid\nabla\zeta\mid^2\mid u_{\Gamma}\mid^2\mathrm{e}^{\zeta}\mathrm{d}S-2\int_{\Gamma}\langle \nabla_{\Gamma}\zeta, \nabla _{\Gamma}u_{\Gamma} \rangle_{\Gamma}\mathrm{e}^{\zeta}\mathrm{d}S-2 \int_{\Gamma}\mid\nabla_{\Gamma}u_{\Gamma}\mid^2 \mathrm{e}^{\zeta}\mathrm{d}S\\
    &\leq \dfrac{-1}{2}\int_{\Omega}\left( \mid\nabla \zeta\mid^2\mid u\mid^2+4\nabla \zeta\cdot \nabla u u+4\mid\nabla u\mid^2 \right)\mathrm{e}^{\zeta}\mathrm{d}x\\
    &-\dfrac{1}{2}\int_{\Gamma}\left( \mid\nabla \zeta\mid^2\mid u_{\Gamma}\mid^2+4\langle\nabla_{\Gamma}\zeta , \nabla_{\Gamma} u_{\Gamma}\rangle_{\Gamma} u_{\Gamma}+4\mid\nabla_{\Gamma} u_{\Gamma}\mid^2 \right)\mathrm{e}^{\zeta}\mathrm{d}S.
\end{align*}
Sine $\mid\nabla_{\Gamma}\zeta\mid^2\leq \mid\nabla \zeta\mid^2$, then
\begin{align*}
    I^{\prime}(t)&\leq \dfrac{-1}{2}\int_{\Omega}\left( \mid\nabla \zeta\mid^2\mid u\mid^2+4\nabla \zeta\cdot \nabla u u+4\mid\nabla u\mid^2 \right)\mathrm{e}^{\zeta}\mathrm{d}x\\
    &-\dfrac{1}{2}\int_{\Gamma}\left( \mid\nabla_{\Gamma} \zeta\mid^2\mid u_{\Gamma}\mid^2+4\langle\nabla_{\Gamma}\zeta , \nabla_{\Gamma} u_{\Gamma}\rangle_{\Gamma} u_{\Gamma}+4\mid\nabla_{\Gamma} u_{\Gamma}\mid^2 \right)\mathrm{e}^{\zeta}\mathrm{d}S\\
    &\leq \dfrac{-1}{2}\bigg(  \int_{\Omega}\mid u\nabla \zeta+2\nabla u\mid^2 \mathrm{e}^{\zeta}\mathrm{d}x+\int_{\Gamma}\mid u_{\Gamma}\nabla_{\Gamma} \zeta+2\nabla_{\Gamma} u_{\Gamma}\mid^2 \mathrm{e}^{\zeta}\mathrm{d}S \bigg)\\
    &\leq 0.
\end{align*}
Since $t\longmapsto I(t)$ is a decreasing function, then 
\begin{align*}
    &\int_{\Omega}\mid u(x,T)\mid^2 \mathrm{e}^{\frac{-\mid x-x_{0}\mid^2}{2\varepsilon}}\mathrm{d}x+\int_{\Gamma}\mid u_{\Gamma}(x,T)\mid^2\mathrm{e}^{\frac{-\mid x-x_{0}\mid^2}{2\varepsilon}}\mathrm{d}S\\
    & \leq\int_{\Omega}\mid u(x,t)\mid^2\mathrm{e}^{\frac{-\mid x-x_{0}\mid^2}{2(T-t+\varepsilon)}}\mathrm{d}x+\int_{\Gamma}\mid u_{\Gamma}(x,t)\mid^2\mathrm{e}^{\frac{-\mid x-x_{0}\mid^2}{2(T-t+\varepsilon)}}\mathrm{d}S.
\end{align*}
It implies that
\begin{align*}
    &\|u(\cdot,T)\|^{2}_{L^2(\Omega \cap B_{x_{0},R})}+\|u_{\Gamma}(\cdot,T)\|^{2}_{L^2(\Gamma \cap B_{x_{0},R})}\\
    &\leq\mathrm{e}^{\frac{R^2}{2\varepsilon}}\int_{\Omega \cap B_{x_{0},R}}\mid u(x,T)\mid^2\mathrm{e}^{\frac{-\mid x-x_{0}\mid^2}{2\varepsilon}}\mathrm{d}x+\mathrm{e}^{\frac{R^2}{2\varepsilon}}\int_{\Gamma \cap B_{x_{0},R}}\mid u_{\Gamma}(x,T)\mid^2\mathrm{e}^{\frac{-\mid x-x_{0}\mid^2}{2\varepsilon}}\mathrm{d}S\\
    &\leq\mathrm{e}^{\frac{R^2}{2\varepsilon}}\int_{\Omega}\mid u(x,t)\mid^2\mathrm{e}^{\frac{-\mid x-x_{0}\mid^2}{2(T-t+\varepsilon)}}\mathrm{d}x+\mathrm{e}^{\frac{R^2}{2\varepsilon}}\int_{\Gamma}\mid u_{\Gamma}(x,t)\mid^2\mathrm{e}^{\frac{-\mid x-x_{0}\mid^2}{2(T-t+\varepsilon)}}\mathrm{d}S\\
    &\leq\mathrm{e}^{\frac{R^2}{2\varepsilon}}\int_{\Omega\cap B_{x_{0},(1+\delta)R}}\mid u(x,t)\mid^2\mathrm{e}^{\frac{-\mid x-x_{0}\mid^2}{2(T-t+\varepsilon)}}\mathrm{d}x\\
    &+\mathrm{e}^{\frac{R^2}{2\varepsilon}}\int_{\Omega\setminus(\Omega\cap B_{x_{0},(1+\delta)R})}\mid u(x,t)\mid^2\mathrm{e}^{\frac{-\mid x-x_{0}\mid^2}{2(T-t+\varepsilon)}}\mathrm{d}x\\
    &+\mathrm{e}^{\frac{R^2}{2\varepsilon}}\int_{\Gamma\cap B_{x_{0},(1+\delta)R}}\mid u_{\Gamma}(x,t)\mid^2\mathrm{e}^{\frac{-\mid x-x_{0}\mid^2}{2(T-t+\varepsilon)}}\mathrm{d}S\\
    &+\mathrm{e}^{\frac{R^2}{2\varepsilon}}\int_{\Gamma\setminus\Gamma\cap B_{x_{0},(1+\delta)R})}\mid u_{\Gamma}(x,t)\mid^2\mathrm{e}^{\frac{-\mid x-x_{0}\mid^2}{2(T-t+\varepsilon)}}\mathrm{d}S\\
    &\leq\mathrm{e}^{\frac{R^2}{2\varepsilon}}\|u(\cdot,t)\|^{2}_{L^2(\Omega\cap B_{x_{0},(1+\delta)R})}+\mathrm{e}^{\frac{R^2}{2\varepsilon}}\|u_{\Gamma}(\cdot,t)\|^{2}_{L^2(\Gamma\cap B_{x_{0},(1+\delta)R})}\\
    &+\mathrm{e}^{\frac{R^2}{2\varepsilon}} \mathrm{e}^{\frac{-R^2(1+\delta)^2}{2(T-t+\varepsilon)}}\|U_{0}\|^2.
\end{align*}
Choose $\dfrac{T}{2}\leq T-\varepsilon\delta\leq t\leq T$ with $0<\varepsilon\leq \dfrac{T}{2}$ and $\delta \in (0,1)$. We obtain
%\begingroup
%\Large
\begin{align*}
    \mathrm{e}^{\frac{R^2}{2\varepsilon}} \mathrm{e}^{\frac{-R^2(1+\delta)^2}{2(T-t+\varepsilon)}}&=\mathrm{e}^{\frac{R^2}{2\varepsilon}-\frac{-R^2(1+\delta)^2}{2(T-t+\varepsilon)}}\\
    &=\mathrm{e}^{\frac{R^2(T-t+\varepsilon)-R^2\varepsilon-2R^2\varepsilon\delta-R^2\varepsilon\delta^2)}{2\varepsilon(T-t+\varepsilon)}}\\
    &=\mathrm{e}^{\frac{R^2[(T-\varepsilon\delta)-t]}{2\varepsilon(T-t+\varepsilon)}}\mathrm{e}^{\frac{-R^2\varepsilon\delta(1+\delta)}{2\varepsilon(T-t+\varepsilon)}}.
\end{align*}
%\endgroup
Since $T-\varepsilon\delta\leq t$, then $\mathrm{e}^{\frac{R^2[(T-\varepsilon\delta)-t]}{2\varepsilon(T-t+\varepsilon)}}\leq 1$ and $\varepsilon(1+\delta)\geq T-t+\varepsilon$, hence 
\[ \dfrac{-R^2\varepsilon\delta(1+\delta)}{2\varepsilon(T-t+\varepsilon)}\leq \dfrac{-R^2\delta}{2\varepsilon}. \]
Thus
\[ \mathrm{e}^{\frac{R^2}{2\varepsilon}} \mathrm{e}^{\frac{-R^2(1+\delta)^2}{2(T-t+\varepsilon)}}\leq\mathrm{e}^{\frac{-R^2\delta}{2\varepsilon}}  .\]
Then, we obtain that
\begin{align}\label{Inn4}
    &\|u(\cdot,T)\|_{L^2(\Omega\cap B_{x_{0},R})}^{2}+\|u_{\Gamma}(\cdot,T)\|_{L^2(\Gamma\cap B_{x_{0},R_{0}})}^2\nonumber\\
    &\leq \mathrm{e}^{\frac{R^2}{2\varepsilon}}\|u(\cdot,T)\|_{L^2(\Omega\cap B_{x_{0},(1+\delta)R})}^{2}+\mathrm{e}^{\frac{R^2}{2\varepsilon}}\|u_{\Gamma}(\cdot,T)\|_{L^2(\Gamma\cap B_{x_{0},(1+\delta)R})}^{2}\nonumber\\
    &+\mathrm{e}^{\frac{-\delta R^2}{2\varepsilon}}\|U_{0}\|^2.
\end{align}
Choose
\[ \varepsilon=\dfrac{\delta R^2}{2\ln \bigg( 2\mathrm{e}{R^2\left( 1+\frac{1}{T} \right)}\frac{\|U_{0}\|^2}{\|u(\cdot,T)\|_{L^2(\Omega\cap B_{x_{0},R)}}^2+\|u_{\Gamma}(\cdot,T)\|_{L^2(\Gamma\cap B_{x_{0},R})}^2} \bigg)} .\]
That is 
\[
\mathrm{e}^{\frac{-\delta R^2}{2\varepsilon}}\|U_{0}\|^2=\dfrac{1}{2}\mathrm{e}^{-R^2\left( 1+\frac{1}{T} \right)}\left(\|u(\cdot,T)\|_{L^2(\Omega\cap B_{x_{0},R})}^2+\|u_{\Gamma}(\cdot,T)\|_{L^2(\Gamma\cap B_{x_{0},R})}^2 \right).
\]
By using \eqref{Inn4}, we obtain
\begin{align*}
    &\bigg(  1-\dfrac{1}{2}\mathrm{e}^{-R^2\left( 1+\frac{1}{T} \right)} \bigg)\bigg( \|u(\cdot,T)\|_{L^2(\Omega\cap B_{x_{0},R})}^2+\|u_{\Gamma}(\cdot,T)\|_{L^2(\Gamma\cap B_{x_{0},R})}^2  \bigg)\\
    &\leq \mathrm{e}^{\frac{R^2}{2\varepsilon}}\|u(\cdot,T)\|_{L^2(\Omega\cap B_{x_{0},(1+\delta)R})}^{2}+\mathrm{e}^{\frac{R^2}{2\varepsilon}}\|u_{\Gamma}(\cdot,T)\|_{L^2(\Gamma\cap B_{x_{0},(1+\delta)R})}^{2},
\end{align*}
and 
\begin{align*}
    \mathrm{e}^{\frac{-\delta R^2}{2\varepsilon}}\|U_{0}\|^2 \leq& \dfrac{1}{2}\mathrm{e}^{-R^2\left( 1+\frac{1}{T} \right)}\bigg(  1-\dfrac{1}{2}\mathrm{e}^{-R^2\left( 1+\frac{1}{T} \right)} \bigg)^{-1}\mathrm{e}^{\frac{R^2}{2\varepsilon}}\\
    &\left( \|u(\cdot,t)\|_{L^2(\Omega\cap B_{x_{0},(1+\delta)R})}^2+\|u_{\Gamma}(\cdot,t)\|_{L^2(\Gamma\cap B_{x_{0},(1+\delta)R})}^2 \right).
\end{align*}
Since $\dfrac{1}{2}\mathrm{e}^{-R^2\left( 1+\frac{1}{T} \right)}\leq \dfrac{1}{2}$, then $\bigg(  1-\dfrac{1}{2} \mathrm{e}^{-R^2\left( 1+\frac{1}{T} \right)}\bigg)^{-1}\leq 2$, thus
\[
\dfrac{1}{2}\mathrm{e}^{-R^2\left( 1+\frac{1}{T} \right)}\bigg(  1-\dfrac{1}{2}\mathrm{e}^{-R^2\left( 1+\frac{1}{T} \right)} \bigg)^{-1}\leq 1
\]
Consequently,
\begin{equation*}
    \dfrac{\|U_{0}\|^2}{\|u(\cdot,t)\|_{L^2(\Omega\cap B_{x_{0},(1+\delta)R})}^2+\|u_{\Gamma}(\cdot,t)\|_{L^2(\Gamma\cap B_{x_{0},(1+\delta)R)}}^2}\leq \mathrm{e}^{(1+\delta)\delta \frac{R^2}{2\theta}},
\end{equation*}
with $\dfrac{1}{\theta}=\dfrac{1}{\varepsilon\delta}$.
\end{proof}
Let us get back to our goal to estimate the term $\dfrac{\|\mathrm{e}^{\Phi/2}G\|_{0}^2}{\|F\|_{0}^2}$.
\begin{align*}
    \dfrac{\|\mathrm{e}^{\Phi/2}G\|_{0}^2}{\|F\|_{0}^2}&=\dfrac{\|\mathrm{e}^{\Phi/2}g\|_{L^2(\Omega\cap B_{x_{0},R_{0}})}^{2}+\|\mathrm{e}^{\Phi/2}g_{\Gamma}\|_{L^2(\Gamma\cap B_{x_{0},R_{0}})}^{2}}{\|f\|_{L^2(\Omega\cap B_{x_{0},R_{0}})}^2+\|f_{\Gamma}\|_{L^2(\Gamma\cap B_{x_{0},R_{0}})}^{2}}\\
    &=\dfrac{ \displaystyle\int_{\Omega\cap B_{x_{0},(1+2\delta)R}}\mid-2\nabla\chi\cdot \nabla u-\Delta\chi u\mid^2 \mathrm{e}^{\Phi}\mathrm{d}x}{\displaystyle\int_{\Omega\cap B_{x_{0},(1+2\delta)R}}\mid\chi u\mid^2\mathrm{e}^{\Phi}\mathrm{d}x+\int_{\Gamma\cap B_{x_{0},(1+2\delta)R}}\mid\chi u_{\Gamma}\mid^2\mathrm{e}^{\Phi}\mathrm{d}S}\\
    &+\dfrac{\displaystyle\int_{\Gamma\cap B_{x_{0},(1+2\delta)R}}  \mid-2\langle \nabla_{\Gamma}\chi,\nabla_{\Gamma}u_{\Gamma} \rangle_{\Gamma}-\Delta_{\Gamma}\chi u_{\Gamma}+\partial_{\nu}\chi u_{\Gamma}\mid^2\mathrm{d}S}{\displaystyle\int_{\Omega\cap B_{x_{0},(1+2\delta)R}}\mid\chi u\mid^2\mathrm{e}^{\Phi}\mathrm{d}x+\int_{\Gamma\cap B_{x_{0},(1+2\delta)R}}\mid\chi u_{\Gamma}\mid^2\mathrm{e}^{\Phi}\mathrm{d}S}.
\end{align*}
Using the fact that $\chi=1$ on $\left\lbrace x; \;\; \mid x-x_{0}\mid\leq (1+3\delta/2)R \right\rbrace$, we obtain
\begin{align}\label{Inn5}
    &\dfrac{\|\mathrm{e}^{\Phi/2}G\|_{0}^2}{\|F\|_{0}^2}\leq \dfrac{ \displaystyle\int_{\Omega\cap \lbrace (1+3\delta/2)R\leq \mid x-x_{0}\mid\leq R_{0} \rbrace}\mid-2\nabla\chi\cdot \nabla u-\Delta\chi u\mid^2 \mathrm{e}^{\Phi}\mathrm{d}x}{\displaystyle\int_{\Omega\cap B_{x_{0},(1+\delta)R}}\mid\chi u\mid^2\mathrm{e}^{\Phi}\mathrm{d}x+\int_{\Gamma\cap B_{x_{0},(1+\delta)R}}\mid\chi u_{\Gamma}\mid^2\mathrm{e}^{\Phi}\mathrm{d}S}\nonumber\\
     &+\dfrac{\displaystyle\int_{\Gamma\cap \lbrace (1+3\delta/2)R\leq \mid x-x_{0}\mid\leq R_{0} \rbrace}  \mid-2\langle \nabla_{\Gamma}\chi,\nabla_{\Gamma}u_{\Gamma} \rangle_{\Gamma}-\Delta_{\Gamma}\chi u_{\Gamma}+\partial_{\nu}\chi u_{\Gamma}\mid^2\mathrm{d}S}{\displaystyle\int_{\Omega\cap B_{x_{0},(1+2\delta)R}}\mid\chi u\mid^2\mathrm{e}^{\Phi}\mathrm{d}x+\int_{\Gamma\cap B_{x_{0},(1+2\delta)R}}\mid\chi u_{\Gamma}\mid^2\mathrm{e}^{\Phi}\mathrm{d}S}\nonumber\\
     &\leq \dfrac{C\exp\left( \max\limits_{(1+3\delta/2)R\leq\mid x-x_{0}\mid\leq R_{0}} \Phi(x)\right)}{\exp\left(\min\limits_{\mid x-x_{0}\mid\leq (1+\delta)R}\Phi(x)\right)}\nonumber\\
     &\times\dfrac{\bigg( \|U(\cdot,t)\|^2+ \| \nabla u(\cdot,t) \|_{L^2(\Omega)}^2+\|\nabla_{\Gamma}u_{\Gamma}(\cdot,t)\|_{L^2(\Gamma)}^2 \bigg)}{\left(\|u(\cdot,t)\|_{L^2(\Omega\cap B_{x_{0},(1+\delta)R})}^2+\|u_{\Gamma}(\cdot,t)\|_{L^2(\Gamma\cap B_{x_{0},(1+\delta)R})}^2\right)}\nonumber\\
     &\leq C\exp \bigg[ -\min \limits_{\mid x-x_{0}\mid\leq(1+\delta)R}\Phi(x,t)+\max\limits_{(1+3\delta/2)R\leq\mid x-x_{0}\mid\leq R_{0}}\Phi(x,t) \bigg]\nonumber\\
     &\qquad\qquad\qquad\times\dfrac{\|U(\cdot,t)\|^2+ \| \nabla u(\cdot,t) \|_{L^2(\Omega)}^2+\|\nabla_{\Gamma}u_{\Gamma}(\cdot,t)\|_{L^2(\Gamma)}^2}{\|u(\cdot,t)\|_{L^2(\Omega\cap B_{x_{0},(1+\delta)R})}^2+\|u_{\Gamma}(\cdot,t)\|_{L^2(\Gamma\cap B_{x_{0},(1+\delta)R})}^2}.
\end{align}
Since $t\longmapsto\|U(\cdot,t)\|^2$ is a decreasing function, then 
\begin{equation}\label{Inn6}
    \|U(\cdot,t)\|^2\leq \|U_{0}\|^2.
    \end{equation}
One the other hand, the operator
\begin{equation*}
\mathbf{A}=\begin{pmatrix} \Delta & 0\\ -\partial_\nu & \Delta_\Gamma \end{pmatrix}
\end{equation*}
generates an analytic $C_0$-semigroup on $\mathbb{L}^2$. Then there is a constant $M>0$ such that
\begin{align*}
    \left\Vert (-\mathbf{A})^{\frac{1}{2}} \mathrm{e}^{t\mathbf{A}} \right\Vert\leq \dfrac{M}{t^{\frac{1}{2}}}.
\end{align*}
Hence,
\begin{align*}
\left\Vert (-\mathbf{A})^{\frac{1}{2}} \mathrm{e}^{t\mathbf{A}} U_{0} \right\Vert\leq \dfrac{M}{t^{\frac{1}{2}}}\left\Vert U_{0}  \right\Vert.
\end{align*}
Therefore
\begin{align*}
\left\Vert (-\mathbf{A})^{\frac{1}{2}} U(\cdot,t)\right\Vert^2\leq \dfrac{M^2}{t}\left\Vert U_{0}  \right\Vert^2.
\end{align*}
We compute $\left\Vert (-\mathbf{A})^{\frac{1}{2}} U(\cdot,t)\right\Vert^2$, we have 
\begin{align*}
    \left\Vert (-\mathbf{A})^{\frac{1}{2}} U(\cdot,t)\right\Vert^2&=\left\langle  (-\mathbf{A})^{\frac{1}{2}} U(\cdot,t), (-\mathbf{A})^{\frac{1}{2}} U(\cdot,t) \right\rangle\\
    &=\left\langle -\mathbf{A} U(\cdot,t),U(\cdot,t) \right\rangle\\
    &=\int_{\Omega} -\Delta u(x,t) u(x,t)\mathrm{d}x+\int_{\Gamma}\partial_{\nu}u(x,t)u_{\Gamma}(x,t)\mathrm{d}S\\
    &-\int_{\Gamma}\Delta_{\Gamma}u_{\Gamma}(x,t)u_{\Gamma}(x,t)\mathrm{d}S\\
    &=\int_{\Gamma}\partial_{\nu}u(x,t)u_{\Gamma}(x,t)\mathrm{d}S+\int_{\Omega}\mid\nabla u(x,t)\mid^2\mathrm{d}x\\
    &+\int_{\Gamma}\partial_{\nu}u(x,t)u_{\Gamma}(x,t)\mathrm{d}S+\int_{\Gamma}\mid\nabla_{\Gamma}u_{\Gamma}(x,t)\mid^2\mathrm{d}S\\
    &=\|\nabla u(\cdot,t)\|_{L^2(\Omega)}^2+\|\nabla _{\Gamma}u_{\Gamma}(\cdot,t)\|_{L^2(\Gamma)}^2.
\end{align*}
Consequently, we obtain
\begin{align}\label{Inn7}
    \|\nabla u(\cdot,t)\|_{L^2(\Omega)}^2+\|\nabla _{\Gamma}u_{\Gamma}(\cdot,t)\|_{L^2(\Gamma)}^2\leq \dfrac{M^2}{t}\|U_{0}\|^2.
\end{align}
Combining \eqref{Inn5}, \eqref{Inn6} and \eqref{Inn7}, we obtain 
\begin{align*}
    \dfrac{\|\mathrm{e}^{\Phi/2}G\|_{0}^2}{\|F\|_{0}^2}&\leq \exp \bigg[-\min \limits_{\mid x-x_{0}\mid\leq(1+\delta)R}\Phi(x,t)+\max\limits_{(1+3\delta/2)R\leq\mid x-x_{0}\mid\leq R_{0}}\Phi(x,t) \bigg]\\
    &\times\dfrac{C\left( 1+\frac{1}{T} \right)\|U_{0}\|^2}{\|u(\cdot,t)\|_{L^2(\Omega\cap B_{x_{0},(1+\delta)R})}^2+\|u_{\Gamma}(\cdot,t)\|_{L^2(\Gamma\cap B_{x_{0},(1+\delta)R})}^2}.
\end{align*}
Using Lemma \ref{Lm2.1}, we get
\begin{align}
    \dfrac{\|\mathrm{e}^{\Phi/2}G\|_{0}^2}{\|F\|_{0}^2}\leq \exp \bigg[ &-\min \limits_{\mid x-x_{0}\mid\leq(1+\delta)R}\Phi(x,t)+\max\limits_{(1+3\delta/2)R\leq\mid x-x_{0}\mid\leq R_{0}}\Phi(x,t) \bigg] \nonumber\\
    &\times C\left( 1+\dfrac{1}{t} \right)\mathrm{e}^{(1+\delta)\delta\frac{R^2}{2\theta}}.
\end{align}
On the other hand, we have $\varphi(x)=\dfrac{-\mid x-x_{0}\mid^2}{4}$. Then 
\begin{align*}
    &\max\limits_{(1+3\delta/2)R\leq\mid x-x_{0}\mid\leq R_{0}}\varphi(x,t)-\min \limits_{\mid x-x_{0}\mid\leq(1+\delta)R}\varphi(x,t)\\
    &=\dfrac{-1}{4}\left( 1+\dfrac{3\delta}{2} \right)^2 R^2+\dfrac{1}{4}(1+\delta)^2 R^2\\
    &=\dfrac{R^2}{4}\bigg[  (1+\delta)^2-\left( 1+\dfrac{3}{2}\delta \right)^2 \bigg]
    &<0.
\end{align*}
In order that
\begin{align*}
    &-\min \limits_{\mid x-x_{0}\mid\leq(1+\delta)R}\Phi(x,t)+\max\limits_{(1+3\delta/2)R\leq\mid x-x_{0}\mid\leq R_{0}}\Phi(x,t)+(1+\delta)\delta\dfrac{R^2}{2\theta}\\
    &=\dfrac{-s}{T-t+h}\left\vert \min \limits_{\mid x-x_{0}\mid\leq(1+\delta)R}\varphi(x,t)-\max\limits_{(1+3\delta/2)R\leq\mid x-x_{0}\mid\leq R_{0}}\varphi(x,t) \right\vert+(1+\delta)\delta\dfrac{R^2}{2\theta}.
\end{align*}
For $T-2\ell h\leq t$ and by taking 
\[ h\leq \dfrac{\theta}{(1+2\ell)h (1+\delta)\delta R^2}\left\vert \min \limits_{\mid x-x_{0}\mid\leq(1+\delta)R}\varphi(x,t)-\max\limits_{(1+3\delta/2)R\leq\mid x-x_{0}\mid\leq R_{0}}\varphi(x,t) \right\vert:=\theta C_{\ell,\phi}, \]
we obtain
\begin{align*}
    &-\min \limits_{\mid x-x_{0}\mid\leq(1+\delta)R}\Phi(x,t)+\max\limits_{(1+3\delta/2)R\leq\mid x-x_{0}\mid\leq R_{0}}\Phi(x,t)+(1+\delta)\delta\dfrac{R^2}{2\theta}\\
    &\leq (1+\delta)\delta\dfrac{R^2}{2\theta}-\dfrac{1}{(1+2\ell)h}\left\vert \min \limits_{\mid x-x_{0}\mid\leq(1+\delta)R}\varphi(x,t)-\max\limits_{(1+3\delta/2)R\leq\mid x-x_{0}\mid\leq R_{0}}\varphi(x,t) \right\vert\\
    &\leq 0.
\end{align*}
By using \eqref{Inn7}, we obtain
\begin{align*}
    \dfrac{\left\Vert  \mathrm{e}^{\Phi/2}G(\cdot,t) \right\Vert_{0}^2}{\|F(\cdot,t)\|_{0}^2}\leq C\left( 1+\dfrac{1}{t} \right),
\end{align*}
for any $T-2\ell h\leq t$ and any $h\leq \theta\min\left(C_{\ell,\varphi},\dfrac{1}{2}\right)$.
\smallskip

\textbf{Step 6.}  We choose $t_{3}=T$, $t_{2}=T-\ell h$, $t_{1}=T-2\ell h$ and $\ell>1$, such that $0<2\ell
h<T$ and $h\leq \theta \min\left(C_{\ell,\varphi},\dfrac{1}{2\ell}\right)$. We make use of \eqref{Inn0}, we obtain
\begin{equation*}
    \left( \left\Vert F(\cdot,T-\ell h) \right\Vert_{0}^2 \right)^{1+M}\leq  \left( \left\Vert F(\cdot,T-2\ell h) \right\Vert_{0}^2\right)^M \|F(\cdot,T)\|_{0}^2 \mathrm{e}^{2D_{\ell}},
\end{equation*}
where $M_{\ell}=\dfrac{(\ell+1)^{c_{0}}-1}{1-\left(\frac{\ell+1}{2\ell+1}\right)^{C_{0}}}$. Consequently, we obtain
\begin{align*}
    &\bigg( \left\Vert \chi u(\cdot,T-\ell h)\mathrm{e}^{\Phi(\cdot,T-\ell h)/2}\right\Vert_{L^2\left( \Omega\cap B_{x_{0},(1+2\delta)R} \right)}\\
    &\qquad\qquad\qquad+\left\Vert \chi u_{\Gamma}(\cdot,T-\ell h)\mathrm{e}^{\Phi(\cdot,T-\ell h)/2} \right\Vert_{L^2(\Gamma\cap B_{x_{0},(1+2\delta)R})} \bigg)^{1+M_{\ell}}\\
    &\leq \left\Vert \chi U(\cdot,T-2\ell h)\mathrm{e}^{\Phi(T-2\ell h)/2} \right\Vert_{0}^{M_{\ell}}\left\Vert \chi u(\cdot,T)\mathrm{e}^{\Phi(\cdot,T)/2} \right\Vert_{0}\mathrm{e}^{C_{\ell,T}},
\end{align*}
where $C_{\ell,T}$ is a constant depending in $\ell$ and $T$. By using the assumptions on the function $\chi$ in \eqref{Ass0}, we obtain
\begin{align*}
    &\bigg( \left\Vert  u(\cdot,T-\ell h)\mathrm{e}^{\Phi(\cdot,T-\ell h)/2}\right\Vert_{L^2\left( \Omega\cap B_{x_{0},(1+\delta)R} \right)}\\
    &\qquad\qquad\qquad+\left\Vert  u_{\Gamma}(\cdot,T-\ell h)\mathrm{e}^{\Phi(\cdot,T-\ell h)/2} \right\Vert_{L^2(\Gamma\cap B_{x_{0},(1+\delta)R})} \bigg)^{1+M_{\ell}}\\
    &\leq \left\Vert  U(\cdot,T-2\ell h)\mathrm{e}^{\Phi(T-2\ell h)/2} \right\Vert_{0}^{M_{\ell}}\left\Vert \chi u(\cdot,T)\mathrm{e}^{\Phi(\cdot,T)/2} \right\Vert_{0}\mathrm{e}^{C_{\ell,T}}.
\end{align*}
Therefore
\begin{align}\label{Inn9}
    &\exp \bigg(  \frac{s}{2h}\frac{1+M_{\ell}}{\ell+1}\min\limits_{x\in \overline{\Omega}\cap \overline{B}_{x_{0},(1+\delta)R}}\varphi(x) \bigg)\nonumber\\
    &\qquad\times\bigg(  \left\Vert  u(\cdot,T-\ell h)\right\Vert_{L^2\left( \Omega\cap B_{x_{0},(1+\delta)R} \right)}+\left\Vert  u_{\Gamma}(\cdot,T-\ell h)\right\Vert_{L^2(\Gamma\cap B_{x_{0},(1+\delta)R})} \bigg)^{1+M_{\ell}}\nonumber\\
    &\leq \mathrm{e}^{C_{\ell,T}}\exp\bigg( \frac{s}{2h}\frac{M_{\ell}}{1+2\ell}\max\limits_{x\in\overline{\Omega}\cap\overline{B}_{x_{0},R_{0}}}\varphi(x) \bigg)\|U(\cdot,T-2\ell h)\|^{M_{\ell}}\left\Vert \chi u(\cdot,T)\mathrm{e}^{\Phi(\cdot,T)/2} \right\Vert_{0}.
\end{align}
Let $\omega_{0}$ be a nonempty subset of $B_{x_{0},r}$, we have 
\begin{align*}
    &\left\Vert \chi u(\cdot,T)\mathrm{e}^{\Phi(\cdot,T)/2} \right\Vert_{0}^{2}\nonumber\\
    &=\int_{\Omega\cap B_{x_{0},R_{0}}}\mid\chi(x)u(x,T)\mid^2\mathrm{e}^{\Phi(x,T)}\mathrm{d}x+\int_{\Gamma\cap B_{x_{0},R_{0}}}\mid\chi(x)u_{\Gamma}(x,T)\mid^2\mathrm{e}^{\Phi(x,T)}\mathrm{d}S\nonumber\\
    &=\int_{\omega_{0}}\mid u(x,T)\mid^2\mathrm{e}^{\Phi(x,T)}\mathrm{d}x+\int_{\left.  (\Omega\cap B_{x_{0},R_{0}})\right\backslash \omega_{0}}\mid\chi(x)u(x,T)\mid^2\mathrm{e}^{\Phi(x,T)}\mathrm{d}x\nonumber\\
    &+\int_{\Gamma\cap B_{x_{0},R_{0}}}\mid\chi(x)u_{\Gamma}(x,T)\mid^2\mathrm{e}^{\Phi(x,T)}\mathrm{d}S\nonumber\\
    &\leq \exp\left[ \max\limits_{x\in\overline{\omega_{0}}}
\Phi\left(  x,T\right)  \right]\|u(\cdot,T)\|_{L^2(\omega_{0})}^2\nonumber\\
     &+\exp\left[\max\limits_{x\in\overline{  \Omega\cap B_{x_{0},R_{0}}\left. 
\right\backslash \omega_{0}}}\Phi\left(  x,T\right)  \right]\|\chi u(\cdot,T)\|_{L^2(\left.  (\Omega\cap B_{x_{0},R_{0}})\right\backslash \omega_{0})}^2\nonumber\\
     &+\exp\left[\max\limits_{x\in\Gamma\cap B_{x_{0},R_{0}}}\Phi\left(x,T\right)  \right]\left\Vert
u_{\Gamma}\left(  x,T\right)  \right\Vert_{L^2(\Gamma\cap B_{x_{0},R_{0}})}^2.
\end{align*}
Using $\max\limits_{x\in\overline{  \Omega\cap B_{x_{0},R_{0}}\left. 
\right\backslash \omega_{0}}}\Phi\left(  x,T\right)\leq\max\limits_{x\in\overline{  \Omega}\cap \overline{B}_{x_{0},R_{0}}\left. 
\right\backslash \omega_{0}}\Phi\left(  x,T\right)$, and the fact that
\[
\|\chi u(\cdot,T)\|_{L^2(\left.  (\Omega\cap B_{x_{0},R_{0}})\right\backslash \omega_{0})}^2+\left\Vert
u_{\Gamma}\left(  x,T\right)  \right\Vert_{L^2(\Gamma\cap B_{x_{0},R_{0}})}^2\leq \|U(\cdot,T)\|^2\leq \| U_0 \|^2,
\]
we obtain
\begin{align}\label{Inn10}
    \left\Vert \chi u(\cdot,T)\mathrm{e}^{\Phi(\cdot,T)/2} \right\Vert_{0}^2&\leq \exp\left[ \max\limits_{x\in\overline{\omega_{0}}}
\Phi\left(  x,T\right)  \right]\|u(\cdot,T)\|_{L^2(\omega_{0})}^2\nonumber\\
&+\exp\left[ \max\limits_{x\in\overline{  \Omega}\cap \overline{B}_{x_{0},R_{0}}\left. 
\right\backslash \omega_{0}}\Phi\left(x,T\right) \right]\|U_{0}\|^2 .
\end{align}
Combining \eqref{Inn9} and \eqref{Inn10}, we obtain
\begin{align}\label{Inn11}
    &\bigg(\left\Vert  u(\cdot,T-\ell h)\right\Vert_{L^2\left( \Omega\cap B_{x_{0},(1+\delta)R} \right)}+\left\Vert  u_{\Gamma}(\cdot,T-\ell h)\right\Vert_{L^2(\Gamma\cap B_{x_{0},(1+\delta)R})}\bigg)^{1+M_{\ell}}\nonumber\\
    &\leq \mathrm{e}^{C_{\ell,T}}\exp \bigg[  \frac{s}{2h}\bigg(-\frac{1+M_{\ell}}{\ell+1}\min\limits_{x\in \overline{\Omega}\cap \overline{B}_{x_{0},(1+\delta)R}}\varphi(x) +\frac{M_{\ell}}{1+2\ell}\max\limits_{x\in\overline{\Omega}\cap\overline{B}_{x_{0},R_{0}}}\varphi(x)\nonumber\\
    &\qquad\qquad\qquad+\max\limits_{x\in\overline{\omega_{0}}}
\varphi\left(  x,T\right) \bigg) \bigg] \|U_0\|^{M_{\ell}}\|u(\cdot,T)\|_{L^2(\omega_{0})}\nonumber\\
    &+\mathrm{e}^{C_{\ell,T}}\exp \bigg[  \frac{s}{2h}\bigg(-\frac{1+M_{\ell}}{\ell+1}\min\limits_{x\in \overline{\Omega}\cap \overline{B}_{x_{0},(1+\delta)R}}\varphi(x) +\frac{M_{\ell}}{1+2\ell}\max\limits_{x\in\overline{\Omega}\cap\overline{B}_{x_{0},R_{0}}}\varphi(x)\nonumber\\
    &\qquad\qquad\qquad+\max\limits_{x\in\overline{  \Omega}\cap \overline{B}_{x_{0},R_{0}}\left. 
\right\backslash \omega_{0}}\varphi\left(  x,T\right) \bigg) \bigg] \|U_0\|^{1+M_{\ell}}.
\end{align}
Since $0<\ell h \leq 2\ell h \leq \theta$, then $T-\theta\leq T-\ell h\leq T$, by using Lemma \ref{Lm2.1}, we obtain that
\begin{align}\label{Inn12}
    \|U_{0}\|\leq \mathrm{e}^{\frac{1}{\theta}C_{\delta,R}}\bigg(\left\Vert  u(\cdot,T-\ell h)\right\Vert_{L^2\left( \Omega\cap B_{x_{0},(1+\delta)R} \right)}+\left\Vert  u_{\Gamma}(\cdot,T-\ell h)\right\Vert_{L^2(\Gamma\cap B_{x_{0},(1+\delta)R})}\bigg),
\end{align}
where $C_{\delta,R}=\dfrac{(1+\delta)\delta R^2}{4}$.

Combining \eqref{Inn11}, \eqref{Inn12} and the fact that $\|U(\cdot,T)\| \leq \|U_{0}\|$, we obtain
\begin{align}
    &\bigg(\mathrm{e}^{\frac{-1}{\theta}C_{\delta,R}}\|U(\cdot,T)\|\bigg)^{1+M_{\ell}}\nonumber\\
    &\leq \mathrm{e}^{C_{\ell,T}}\exp \bigg[  \frac{s}{2h}\bigg(-\frac{1+M_{\ell}}{\ell+1}\min\limits_{x\in \overline{\Omega}\cap \overline{B}_{x_{0},(1+\delta)R}}\varphi(x) +\frac{M_{\ell}}{1+2\ell}\max\limits_{x\in\overline{\Omega}\cap\overline{B}_{x_{0},R_{0}}}\varphi(x)\nonumber\\
    &\qquad\qquad\qquad+\max\limits_{x\in\overline{\omega_{0}}}
\varphi\left(  x,T\right) \bigg) \bigg]\times \|U(\cdot,0)\|^{M_{\ell}}\|u(\cdot,T)\|_{L^2(\omega_{0})}\nonumber\\
    &+\mathrm{e}^{C_{\ell,T}}\exp \bigg[  \frac{s}{2h}\bigg(-\frac{1+M_{\ell}}{\ell+1}\min\limits_{x\in \overline{\Omega}\cap \overline{B}_{x_{0},(1+\delta)R}}\varphi(x) +\frac{M_{\ell}}{1+2\ell}\max\limits_{x\in\overline{\Omega}\cap\overline{B}_{x_{0},R_{0}}}\varphi(x)\nonumber\\
    &\qquad\qquad\qquad+\max\limits_{x\in\overline{  \Omega}\cap \overline{B}_{x_{0},R_{0}}\left. 
\right\backslash \omega_{0}}\varphi\left(  x,T\right) \bigg) \bigg] 
     \|U_0\|^{1+M_{\ell}}.
\end{align}
Since $\varphi(x)\leq 0$, $\forall\, x\in \overline{\Omega}$, $x_{0}\in \omega_{0}$ and $\varphi(x_{0})=0$,  then $\max\limits_{x\in\overline{\Omega}\cap\overline{B}_{x_{0},R_{0}}}%
\varphi\left(  x\right)=0$.  
Hence, 
\begin{align*}
    &-\frac{1+M_{\ell}}{\ell+1}\min\limits_{x\in \overline{\Omega}\cap \overline{B}_{x_{0},(1+\delta)R}}\varphi(x) +\frac{M_{\ell}}{1+2\ell}\max\limits_{x\in\overline{\Omega}\cap\overline{B}_{x_{0},R_{0}}}\varphi(x)+\max\limits_{x\in\overline{\omega_{0}}}
\varphi\left(  x,T\right)\\
    &=-\frac{1+M_{\ell}}{\ell+1}\min\limits_{x\in \overline{\Omega}\cap \overline{B}_{x_{0},(1+\delta)R}}\varphi(x) +\max\limits_{x\in\overline{\omega_{0}}}
\varphi\left(  x,T\right)< 0,
\end{align*}
for $\ell \geq 1$ sufficiently large. Therefore, there exist $C_{1} >0$ and $C_{2}>0$ such that for any $h>0$, with $h\leq \theta \min\left( C_{\ell,\varphi},\dfrac{1}{2\ell} \right):=\theta C_{3}$,
\begin{align*}
    \bigg(\mathrm{e}^{\frac{-1}{\theta}C_{\delta,R}}\|U(\cdot,T)\|\bigg)^{1+M_{\ell}}&\leq \mathrm{e}^{C_{\ell,T}} \mathrm{e}^{C_{1}\frac{1}{h}}\left\Vert U_{0}    \right\Vert
^{M_{\ell}}\left\Vert u\left(  \cdot,T\right)  \right\Vert _{L^{2}\left(
\omega_{0}\right)  }\\
&\qquad\qquad\qquad\qquad+ \mathrm{e}^{C_{\ell,T}} \mathrm{e}^{-C_{2}\frac{1}{h}}\left\Vert U_0
\right\Vert ^{1+M_{\ell}}.
\end{align*}
On the other hand, for any $h\geq \theta C_{3}$, we have $\mathrm{e}^{C_{\ell,T}}\mathrm{e}^{\frac{C_{2}}{C_{3}}\frac{1}{\theta}}\mathrm{e}^{-c_{2}\frac{1}{h}}\geq 1$. Therefore
\begin{align*}
    \left( \mathrm{e}^{\frac{-1}{\theta}C_{\delta,R}}\|U(\cdot,T)\| \right)^{1+M_{\ell}}&\leq \|U(\cdot,T)\|^{1+M_{\ell}}\\
    &\leq\|U_{0}\|^{1+M_{\ell}}\\
    &\leq\mathrm{e}^{C_{\ell,T}}\mathrm{e}^{\frac{C_{2}}{C_{3}}\frac{1}{\theta}}\mathrm{e}^{-C_{2}\frac{1}{h}}\|U_{0}\|^{1+M_{\ell}}\\
    &\leq \mathrm{e}^{C_{\ell,T}} \mathrm{e}^{C_{1}\frac{1}{h}}\left\Vert U_{0}  \right\Vert
^{M_{\ell}}\left\Vert u\left(  \cdot,T\right)  \right\Vert _{L^{2}\left(
\omega_{0}\right)}\\
&\qquad\qquad\qquad\qquad+ \mathrm{e}^{C_{\ell,T}} \mathrm{e}^{\frac{C_{2}}{C_{3}}\frac{1}{\theta}}\mathrm{e}^{-C_{2}\frac{1}{h}}\left\Vert U_{0}
\right\Vert ^{1+M_{\ell}}.
\end{align*}
Consequently, for any $h>0$
\begin{align*}
    \left( \mathrm{e}^{\frac{-1}{\theta}C_{\delta,R}}\|U(\cdot,T)\| \right)^{1+M_{\ell}}\leq&\mathrm{e}^{C_{\ell,T}} \mathrm{e}^{C_{1}\frac{1}{h}}\left\Vert U_{0}  \right\Vert
^{M_{\ell}}\left\Vert u\left(  \cdot,T\right)  \right\Vert _{L^{2}\left(
\omega_{0}\right)}\\
&\qquad\qquad\qquad\qquad+ \mathrm{e}^{C_{\ell,T}} \mathrm{e}^{\frac{C_{2}}{C_{3}}\frac{1}{\theta}}\mathrm{e}^{-C_{2}\frac{1}{h}}\left\Vert U_{0}
\right\Vert ^{1+M_{\ell}}.
\end{align*}
We choose $h>0$ such that 
\[
\mathrm{e}^{C_{\ell,T}}\mathrm{e}^{\frac{C_{2}}{C_{3}}\frac{1}{\theta}}\mathrm{e}^{-C_{2}\frac{1}{h}}\left\Vert U_{0}  \right\Vert ^{1+M_{\ell}}=\frac{1}{2}\bigg(\mathrm{e}^{\frac{-1}{\theta}C_{\delta,R}}\left\Vert U\left(
\cdot,T\right)  \right\Vert ^{1+M_{\ell}}\bigg)\text{ .}%
\]
That is,
\[
\mathrm{e}^{C_{2}\frac{1}{h}}=2\mathrm{e}^{C_{\ell,T}}\mathrm{e}^{\frac{C_{2}}{C_{3}}\frac{1}{\theta}}\left(  \frac{\left\Vert
U_{0}  \right\Vert }{\mathrm{e}^{\frac{-1}{\theta}C_{\delta,R}}\left\Vert U\left(  \cdot,T\right)
\right\Vert }\right)^{1+M_{\ell}}.%
\]
It follows that
\begin{align*}
\dfrac{1}{2}\bigg(\mathrm{e}^{\frac{-1}{\theta}C_{\delta,R}}\left\Vert U\left(  \cdot,T\right)  \right\Vert\bigg)^{1+M_{\ell}}&\leq \mathrm{e}^{C_{\ell,T}}\left(
2\mathrm{e}^{C_{\ell,T}}\mathrm{e}^{C_{2}\frac{2\ell}{T}}\left(  \frac{\left\Vert U\left(\cdot,0\right)
\right\Vert }{\left\Vert U\left(\cdot,T\right)  \right\Vert }\right)
^{1+M_{\ell}}\right)  ^{\frac{C_{1}}{C_{2}}}\\
&\qquad\qquad\qquad\qquad\times\left\Vert U_{0}  \right\Vert ^{M_{\ell}}\left\Vert u\left(  \cdot,T\right)
\right\Vert _{L^{2}\left(\omega_0\right)}.
\end{align*}
Therefore,
\begin{align*}
\left(\mathrm{e}^{\frac{-1}{\theta}C_{\delta,R}}\left\Vert U\left(  \cdot,T\right)  \right\Vert\right)^{1+M_{\ell}+\left(  1+M_{\ell}\right)  \frac{C_{1}}{C_{2}}} \leq& 2^{1+\frac{C_{1}}{C_{2}}}%
\mathrm{e}^{C_{\ell,T}\left(1+\frac{C_{1}}{C_{2}}\right)}\mathrm{e}^{\frac{C_{1}}{C_{3}}\frac{1}{\theta}}\\
&\times\left(\left\Vert U_{0} \right\Vert\right)
^{M_{\ell}+\left(  1+M_{\ell}\right)  \frac{C_{1}}{C_{2}}}\left\Vert u\left(
\cdot,T\right)  \right\Vert _{L^{2}\left(  \omega_{0}\right)  }.
\end{align*}
Hence, for any $c^{*} \geq M_{\ell}+(1+M_{\ell})\frac{C_{1}}{C_{2}}$, we obtain
\[
\|U(\cdot,T)\|^{1+c^{*}}\leq 2^{1+\frac{C_{1}}{C_{2}}}%
\mathrm{e}^{C_{\ell,T}\left(1+\frac{C_{1}}{C_{2}}\right)}\mathrm{e}^{\left(\frac{C_{1}}{C_{3}}+c^{*}C_{\delta,R}\right)\frac{1}{\theta}}\|U_{0}\|^{c^{*}}\|u(\cdot,T)\|_{L^2(\omega_{0})},
\]
Then, for some $c>0$, we obtain
\[
\|U(\cdot,T)\|^{1+c}\leq c \mathrm{e}^{c\frac{1}{\theta}}\|U_{0}\|^c\|u(\cdot,T)\|_{L^2(\omega_{0})}.
\]
Recall the definition of $\theta$ in Lemma
\ref{Lm2.1},
\begin{align*}
\dfrac{1}{\theta}=\dfrac{2}{(\delta R)^2}\ln \left( 2 \mathrm{e}^{R^2\left(1+\frac{1}{T}\right)}\dfrac{\|U_{0}\|^2}{\|u(\cdot,T)\|_{L^2(\Omega \cap B_{x_{0},R})}^2+\|u_{\Gamma}(\cdot,T)\|_{L^2(\Gamma\cap B_{x_{0},R})}^2} \right).
\end{align*}
Consequently, we obtain
\begin{align*}
\|U(\cdot,T)\|^{1+c}\leq& c \bigg[ 2 \mathrm{e}^{R^2\left(1+\frac{1}{T}\right)}\dfrac{\|U_{0}\|^2}{\|u(\cdot,T)\|_{L^2(\Omega \cap B_{x_{0},R})}^2+\|u_{\Gamma}(\cdot,T)\|_{L^2(\Gamma\cap B_{x_{0},R})}^2} \bigg]\\
&\times\|U_{0}\|^{c}\|u(\cdot,T)\|_{L^2(\omega_{0})}.
\end{align*}
Hence, for some $k>0$, we obtain
\[
\left(\|u(\cdot,T)\|_{L^2(\Omega \cap B_{x_{0},R})}^2+\|u_{\Gamma}(\cdot,T)\|_{L^2(\Gamma\cap B_{x_{0},R})}^2\right)^{1+k}\leq c\mathrm{e}^{k\left( 1+\frac{1}{T} \right)}\|U_{0}\|^k\|u(\cdot,T)\|_{L^2(\omega_{0})}.
\]
This  provides  the desired inequality.\endproof
\subsection{Proof of Theorem \ref{theo1.1}} First, we show that, for any compact sets $\Theta_{1}$ and $\Theta_{2}$ with non-empty interior in $\Omega$ there are constants $C>0$ and $\beta\in (0,1)$ such that 
\begin{equation}\label{Inn13}
\|u(\cdot,T)\|_{L^2(\Theta_{1})}\leq \mathrm{e}^{C\left(  1+\frac{1}{T}\right)} \|u(\cdot,T)\|_{L^2(\Theta_{1})}^{\sigma_1}\|U_{0}\|^{1-\sigma_1}.
\end{equation}
Indeed, since $\Theta_1$ is a compact in $\Omega$, there are $R>0$ and finitely many points $x_{1},\ldots,x_{m}$ such that $\Theta_{1}\subset \bigcup\limits_{i=1,\ldots,m}B(x_{i},R)$ and $B(x_{i},R)\subset\Omega$, for any $i \in \lbrace 1,\ldots,m \rbrace$. Next, for each $i\in \lbrace 1,\ldots,m \rbrace$, we choose $q\in (0,R)$ and finitely many points $a_{0},\cdots,a_{l}$ with the following properties:
\[
\left\{\begin{array}{l}
x_{i}=a_{l} \\
\Theta_{2} \supset B\left(a_{0}, q\right) \\
B\left(a_{j+1}, q / 2\right) \subset B\left(a_{j}, q\right) \quad \forall j=0, \ldots, l-1, \\
B\left(a_{j}, 3 q\right) \subset \Omega \quad \forall j=0, \ldots, l .
\end{array}\right.
\]
Using Theorem \eqref{theo1.2}, there are $\sigma_1,\beta_l, \ldots, \beta_{0}$ such that 
\begin{align*}
    \|u(\cdot,T)\|_{L^2(B_{x_i,R})}&\leq \mathrm{e}^{C\left(1+\frac{1}{T}\right)} \|u(\cdot,T)\|_{L^2\left(B_{x_i,\frac{q}{2}}\right)}^{\beta_l}\|U_0\|^{1-\beta_{l}}\\
    &\leq \mathrm{e}^{C\left(1+\frac{1}{T}\right)} \|u(\cdot,T)\|_{L^2\left(B_{a_l,\frac{q}{2}}\right)}^{\beta_l}\|U_0\|^{1-\beta_{l}}\\
    &\leq\mathrm{e}^{C\left(1+\frac{1}{T}\right)}\|u(\cdot,T)\|_{L^2\left(B_{a_l,q}\right)}^{\beta_l}\|U_0\|^{1-\beta_{l}}\\
    &\leq\mathrm{e}^{C\left(1+\frac{1}{T}\right)}\|U_0\|^{1-\beta_{l}}\bigg( \mathrm{e}^{C\left(1+\frac{1}{T}\right)}\|u(\cdot,T)\|_{L^2\left(B_{a_{l-1},\frac{q}{2}}\right)}^{\beta_{l-1}}\|U_0\|^{1-\beta_{l-1}}\bigg)^{\beta_l}\\
    &\leq\mathrm{e}^{2C\left(1+\frac{1}{T}\right)}\|u(\cdot,T)\|_{L^2\left(B_{a_{l-1},\frac{a}{2}} \right)}^{\beta_{l}\beta_{l-1}}\|U_0\|^{1-\beta_{l}\beta_{l-1}}\\
    & \;\; \vdots\\
    &\leq\mathrm{e}^{C\left(1+\frac{1}{T}\right)}\|u(\cdot,T)\|_{L^2\left(B_{a_0,q}\right)}^{\sigma_1}\|U_0\|^{1-\sigma_{1}}\\
    &\leq\mathrm{e}^{C\left(1+\frac{1}{T}\right)}\|u(\cdot,T)\|_{L^2\left(\Theta_2\right)}^{\sigma_1}\|U_0\|^{1-\sigma_{1}},
\end{align*}
This implies the inequality \eqref{Inn13}.

Next, since $\Omega$ is a bounded domain with a $C^2$ boundary. Then, there is a finite set of triplet $(a_{j},R_j,\delta_j)\in \Omega \times \mathbb{R}_{+}^*\times(0,1]$, $j=1,\ldots,m$, such that 
\begin{equation*}
    \partial \Omega \subset \bigcup_{j=1, \ldots, m} B_{a_{j},\left(1+2 \delta_{j}\right) R_{j}},
\end{equation*}
and $\Omega \cap B_{a_{j},(1+2\delta_{j})R_j}$ is star-shaped with respect to $a_j$ for some $\delta_j$. Hence we apply Theorem \eqref{theo1.2} for $j=1,\ldots,m$. Then, when $\mathcal{V}$ is a neighborhood of $\Gamma$ and $\Theta_{3}$ is a compact set with nonempty interior in $\Omega$, there exist constants $C>0$ and $\sigma_2\in(0,1)$ such that 
\begin{align}\label{Inn14}
    \left\Vert u(\cdot,T) \right\Vert_{L^2(\mathcal{V}\left\backslash \Gamma \right.)}+\|u_{\Gamma}(\cdot,T)\|_{L^2(\Gamma)}\leq \mathrm{e}^{C\left( 1+\frac{1}{T} \right)}\|u(\cdot,T)\|_{L^2(\Theta_{3})}^{\sigma_2}\|U_{0}\|^{1-\sigma_{2}}.
\end{align}
By using \eqref{Inn13}, we obtain
\begin{equation*}
\|u(\cdot,T)\|_{L^2(\Theta_{1})}\|U_0\|^{\sigma_1}\leq \mathrm{e}^{C\left(  1+\frac{1}{T}\right)} \|u(\cdot,T)\|_{L^2(\Theta_{1})}^{\sigma_1}\|U_{0}\|.
\end{equation*}
Therefore
\begin{equation}\label{Inn15}
\|u(\cdot,T)\|_{L^2(\Theta_{1})}^{\frac{1}{\sigma^1}}\|U_0\|\leq \mathrm{e}^{\frac{C}{\sigma_1}\left(  1+\frac{1}{T}\right)} \|u(\cdot,T)\|_{L^2(\Theta_{1})}\|U_{0}\|^{\frac{1}{\sigma_1}}.
\end{equation}
By the same reasoning, we apply \eqref{Inn14} to obtain
\begin{align}\label{Inn16}
    \left(\left\Vert u(\cdot,T) \right\Vert_{L^2(\mathcal{V}\left\backslash \Gamma \right.)}+\|u_{\Gamma}(\cdot,T)\|_{L^2(\Gamma)}\right)^{\frac{1}{\sigma_2}}\|U_0\|\leq \mathrm{e}^{\frac{C}{\sigma_2}\left( 1+\frac{1}{T} \right)}\|u(\cdot,T)\|_{L^2(\Theta_{3})}\|U_{0}\|^{\frac{1}{\sigma_2}}.
\end{align}
We pose $\alpha=\max \left( \frac{1}{\sigma_1},\frac{1}{\sigma_2} \right)$ and multiplying \eqref{Inn15} by $\|u(\cdot,T)\|_{L^2(\Theta_1)}^{\alpha-\frac{1}{\sigma_1}}$, 
\begin{equation*}
\|u(\cdot,T)\|_{L^2(\Theta_{1})}^{\alpha}\|U_0\|\leq \mathrm{e}^{\alpha C\left(  1+\frac{1}{T}\right)} \|u(\cdot,T)\|_{L^2(\Theta_{1})}\|U_{0}\|^{\frac{1}{\sigma_1}}\|u(\cdot,T)\|_{L^2(\Theta_1)}^{\alpha-\frac{1}{\sigma_1}}.
\end{equation*}
Since $\|u(\cdot,T)\|_{L^2(\Theta_1)}\leq \|U(\cdot,T)\|\leq \|U_0\|$, then 
\begin{equation}\label{Inn17}
\|u(\cdot,T)\|_{L^2(\Theta_{1})}^{\alpha}\|U_0\|\leq \mathrm{e}^{\alpha C\left(  1+\frac{1}{T}\right)} \|u(\cdot,T)\|_{L^2(\Theta_{1})}\|U_{0}\|^{\alpha}.
\end{equation}
Similarly, we multiply \eqref{Inn16} by $\left( \|u(\cdot,T)\|_{L^2(\mathcal{V}\left\backslash \Gamma \right.)}+\|u_\Gamma(\cdot,T)\|_{L^2(\Gamma)} \right)^{\alpha-\frac{1}{\sigma_2}}$,
\begin{align}\label{Inn18}
    \left(\left\Vert u(\cdot,T) \right\Vert_{L^2(\mathcal{V}\left\backslash \Gamma \right.)}+\|u_{\Gamma}(\cdot,T)\|_{L^2(\Gamma)}\right)^{\alpha}&\|U_0\|\leq \mathrm{e}^{\alpha C\left( 1+\frac{1}{T} \right)}\|u(\cdot,T)\|_{L^2(\Theta_{3})}\|U_{0}\|^{\frac{1}{\sigma_2}}\nonumber\\
    &\times \left(\left\Vert u(\cdot,T) \right\Vert_{L^2(\mathcal{V}\left\backslash \Gamma \right.)}+\|u_{\Gamma}(\cdot,T)\|_{L^2(\Gamma)}\right)^{\alpha-\frac{1}{\sigma_2}}.
\end{align}
Next, we make use of \eqref{Inn16}, we obtain
\begin{align}\label{Inne18}
    \left\Vert u(\cdot,T) \right\Vert_{L^2(\mathcal{V}\left\backslash \Gamma \right.)}+\|u_{\Gamma}(\cdot,T)\|_{L^2(\Gamma)}\leq \mathrm{e}^{C\left( 1+\frac{1}{T} \right)}\|U_0\|.
\end{align}
By \eqref{Inn18} and \eqref{Inne18}, we obtain that
\begin{align}\label{Inn19}
    \left(\left\Vert u(\cdot,T) \right\Vert_{L^2(\mathcal{V}\left\backslash \Gamma \right.)}+\|u_{\Gamma}(\cdot,T)\|_{L^2(\Gamma)}\right)^{\alpha}&\|U_0\|\leq \mathrm{e}^{(\alpha+1) C\left( 1+\frac{1}{T} \right)}\|u(\cdot,T)\|_{L^2(\Theta_{3})}\|U_{0}\|^{\alpha}.
\end{align}
Combining \eqref{Inn17} and \eqref{Inn19}, we obtain
\begin{align*}
    &\left[ \|u(\cdot,T)\|_{L^2(\Theta_1)}^\alpha+ \left(\left\Vert u(\cdot,T) \right\Vert_{L^2(\mathcal{V}\left\backslash \Gamma \right.)}+\|u_{\Gamma}(\cdot,T)\|_{L^2(\Gamma)}\right)^{\alpha}\right]\|U_0\|\\
    &\leq \mathrm{e}^{(\alpha+1) C\left( 1+\frac{1}{T} \right)}\left( \|u(\cdot,T)\|_{L^2(\Theta_2)}+\|u(\cdot,T)\|_{L^2(\Theta_3)} \right)\|U_0\|^\alpha.
\end{align*}
Consequently,
\begin{align*}
    &\left(\|u(\cdot,T)\|_{L^2(\Theta_1)}+\left\Vert u(\cdot,T) \right\Vert_{L^2(\mathcal{V}\left\backslash \Gamma \right.)}+\|u_{\Gamma}(\cdot,T)\|_{L^2(\Gamma)}\right)^{\alpha}\|U_0\|\\
    &\leq \mathrm{e}^{(\alpha+1) C\left( 1+\frac{1}{T} \right)}\left( \|u(\cdot,T)\|_{L^2(\Theta_2)}+\|u(\cdot,T)\|_{L^2(\Theta_3)} \right)\|U_0\|^\alpha.
\end{align*}
Since $\overline{\Omega}\subset \Theta\cup \mathcal{V}$ and $\Theta_2\cup \Theta_3\subset \omega$, then
\begin{align*}
    \|U(\cdot,T)\|^\alpha \|U_0\|\leq \mathrm{e}^{(\alpha+1)C\left( 1+\frac{1}{T} \right)}\|u(\cdot,T)\|_{L^2(\omega)}\|U_0\|^{\alpha}.
\end{align*}
Using the fact that $\|U(\cdot,T)\|\leq \|U_0\|$, we obtain
\begin{align*}
    \|U(\cdot,T)\|^{\alpha+1} \leq \mathrm{e}^{(\alpha+1)C\left( 1+\frac{1}{T} \right)}\|u(\cdot,T)\|_{L^2(\omega)}\|U_0\|^{\alpha}.
\end{align*}
Finally,
\begin{align*}
    \|U(\cdot,T)\| \leq \mathrm{e}^{C\left( 1+\frac{1}{T} \right)}\|u(\cdot,T)\|_{L^2(\omega)}^{\frac{1}{\alpha+1}}\|U_0\|^{\frac{\alpha}{\alpha+1}}.
\end{align*}
This completes the proof.

\section{Finite time stabilization}\label{sec4}
Let us start with an estimate that follows from the Weyl asymptotic formula. Let $\Phi_k = \left(\phi_k , \phi_{\Gamma,k} \right) \in \mathbb{L}^2$ be the family of orthonormal eigenfunctions of the operator $-\mathbf{A}$ corresponding to the eigenvalues $(\lambda_k)_{k\geq1}$ given by Lemma \ref{lmsp1}. It holds that $\lambda_k \sim C(\Omega) k^{\frac{2}{n}}$, see \cite[Theorem 2.16]{Ci'16} for a detailed formula. Then there is a positive constant $C=C(\Omega)>0$ such that
\begin{equation}\label{equa34}
   \mathrm{card}\left\lbrace \lambda_{i}\leq \Lambda\right\rbrace=\sum_{\lambda_{i}\leq \Lambda} 1\leq C \Lambda^{\frac{n}{2}}.  
\end{equation}
Define an increasing sequence $\{t_{k}\}$  converging to $T$ by
\begin{equation}\label{equ34}
    t_{k}=T\left( 1-\dfrac{1}{b^k} \right),
\end{equation}
with $b>1$.
Introduce the linear operator $\mathcal{L}_{k}$ by 
\begin{align}\label{eq35}
    \mathcal{L}_{k}:\;  \mathbb{L}^2 &\rightarrow L^2(\omega) \nonumber\\
                    \vartheta\quad&\mapsto \sum _{\lambda_{i}\leq \Lambda_{k}}\left\langle \vartheta, \Phi_{i} \right\rangle  h_{i},
\end{align}
where $\Lambda_{k}:=\lambda_{1}+\dfrac{\eta}{T}\dfrac{b^{2k+1}}{b-1}$ with $\eta>1$ and $h_{i}$  is the impulse control of the following heat equation associated with the eigenfunction $\Phi_{i}$.
\begin{empheq}[left = \empheqlbrace]{alignat=2}
\begin{aligned}\label{equation28}
&\partial_{t} \psi_{i}-\Delta \psi_{i}=0, &&\text { in } \Omega \times(t_{k}, t_{k+1}) \backslash\{\tau_{k}\},\\
&\psi_{i}(\cdot, \tau_{k})=\psi_{i}\left(\cdot, \tau_{k}^{-}\right)+\mathds{1}_{\omega} h_{i}(\cdot,t_{k}), &&\text { in } \Omega,\\
&\partial_{t}\psi_{i\Gamma} - \Delta_{\Gamma} \psi_{i\Gamma} + \partial_{\nu}\psi_{i} =0, &&\text { on } \Gamma \times(t_{k}, t_{k+1})\backslash\{\tau_{k}\}, \\
&\psi_{i\Gamma}(\cdot, \tau_{k})=\psi_{i\Gamma}\left(\cdot, \tau_{k}^{-}\right), &&\text { on } \Gamma,\\
& \psi_{i\Gamma}(x,t) = \psi_{i\mid\Gamma}(x,t), &&\text{ on } \Gamma \times(t_{k}, t_{k+1}),\\
& \left(\psi_{i}(\cdot, 0),\psi_{i\Gamma}(\cdot, 0)\right)=\Phi_{i} && \text{ on } \Omega\times\Gamma.
\end{aligned}
\end{empheq}
Applying Theorem \ref{thm1.2} for $\varepsilon=\dfrac{\mathrm{e}^{-\eta b^{k}}}{\sum\limits_{\lambda_{i}\leq \Lambda_{k}}1}$ we obtain
\begin{equation}\label{equa29}
    \|(\psi_{i},\psi_{i\Gamma})\|^2\leq \dfrac{\mathrm{e}^{-\eta b^{ k}}}{\sum\limits_{\lambda_{i}\leq \Lambda_{k}}1}
\end{equation}
and
\begin{equation}\label{equa30}
\|h_{i}\|_{L^2(\omega)}^2\leq\mathrm{e}^{C_{3}\left(1+\frac{2}{t_{k+1}-t_{k}}\right)}\mathrm{e}^{\frac{C_{3}\sqrt{2}}{\sqrt{t_{k+1}-t_{k}}}\sqrt{\ln \left( \mathrm{e}+\mathrm{e}^{\eta b^{k}}\sum\limits_{\lambda_{i}\leq \Lambda_{k}}1 \right)}}
\end{equation}
\subsection{Proof of Theorem \ref{thm1.4}}
We start by estimating the solution $\Psi := \left(\psi,\psi_{\Gamma}\right)$ of the system \eqref{1.1}
on the interval $\left(t_{k} , t_{k+1} \right)$ with initial data $\displaystyle \Psi\left(t_{k}\right) := \left(\psi,\psi_{\Gamma}\right)\left(t_{k}\right) = \sum_{i \geq 1} a_{i} \Phi_{i} \in \mathbb{L}^2$. To do so, we consider the following two systems
\begin{empheq}[left = \empheqlbrace]{alignat=2}\label{eq29}
\begin{aligned}
&\partial_{t} \upsilon-\Delta \upsilon=0, &&\text { in } \Omega \times(t_{k}, t_{k+1}) \backslash\{\tau_{k}\},\\
&\upsilon(\cdot, \tau_{k})=\upsilon\left(\cdot, \tau_{k}^{-}\right)+\mathds{1}_{\omega} \sum_{\lambda_{i} \leq \Lambda_{k}} a_{i} h_{i}, &&\text { in } \Omega,\\
&\partial_{t}\upsilon_{\Gamma} - \Delta_{\Gamma} \upsilon_{\Gamma} + \partial_{\nu}\upsilon =0, &&\text { on } \Gamma \times(t_{k}, t_{k+1})\backslash\{\tau_{k}\},\\
&\upsilon_{\Gamma}(\cdot, \tau_{k})=\upsilon_{\Gamma}\left(\cdot, \tau_{k}^{-}\right), &&\text { on } \Gamma,\\
& \upsilon_{\Gamma}(x,t) = \upsilon_{\mid\Gamma}(x,t), &&\text{ on } \Gamma \times(t_{k}, t_{k+1}) ,\\
& \left(\upsilon(\cdot, t_{k}),\upsilon_{\Gamma}(\cdot, t_{k})\right)=\sum_{\lambda_{i} \leq \Lambda_{k}} a_{i} \Phi_{i} && \text{ on } \Omega\times\Gamma.
\end{aligned}
\end{empheq}
and
\begin{empheq}[left = \empheqlbrace]{alignat=2}\label{eq28}
\begin{aligned}
&\partial_{t} \varphi-\Delta \varphi=0, &&\text { in } \Omega \times(t_{k}, t_{k+1}),\\
&\partial_{t}\varphi_{\Gamma} - \Delta_{\Gamma} \varphi_{\Gamma} + \partial_{\nu}\varphi =0, &&\text { on } \Gamma \times(t_{k}, t_{k+1}),\\
& \varphi_{\Gamma}(x,t) = \varphi_{\mid\Gamma}(x,t), &&\text{ on } \Gamma \times(t_{k}, t_{k+1}) ,\\
& \left(\varphi(\cdot, t_{k}),\varphi_{\Gamma}(\cdot, t_{k})\right)= \sum_{\lambda_{i} > \Lambda_{k}} a_{i} \Phi_{i}, && \text{ on } \Omega\times\Gamma.
\end{aligned}
\end{empheq}
The solutions of the above systems are given by
\begin{equation}\label{eq31}
\Upsilon := (\upsilon, \upsilon_{\Gamma}) = \sum_{\lambda_{i} \leq \Lambda_{k}} a_{i} \Psi_{i},  
\end{equation}
where $\Psi_{i}:= (\psi_{i},\psi_{i\Gamma})$ is the solution of the system \eqref{equation28} and
\begin{equation}\label{eq32}
\varrho\left(t\right):= (\varphi,\varphi_{\Gamma})\left(t\right)=\sum_{\lambda_{i}>\Lambda_{k}} a_{i} \mathrm{e}^{-\lambda_{i}\left(t-t_{k}\right)} \Phi_{i} 
\end{equation}
using \eqref{equa29} and \eqref{eq31}, we obtain
\begin{equation}\label{eq42}
\left\|\Upsilon\left(t_{k+1}\right)\right\|^{2} \leq \sum_{\lambda_{i} \leq \Lambda_{k}}\mid a_{i}\mid^{2} \frac{\mathrm{e}^{-\eta b^{ k}}}{\sum_{\lambda_{i} \leq \Lambda_{k}} 1} \leq \mathrm{e}^{- \eta b^{ k}}\left\|\Psi\left(t_{k}\right)\right\|^{2},
\end{equation}
on the other hand, by \eqref{eq32} we obtain
\begin{equation*}
\left\|\varrho\left(t_{k+1}\right)\right\|^{2} \leq e^{-2\Lambda_{k}\left(t_{k+1}-t_{k}\right)}\left\|\Psi\left(t_{k}\right)\right\|^{2},
\end{equation*}
using the fact that $ \Lambda_{k}\left(t_{k+1}-t_{k}\right)= \Lambda_{1}\left(t_{k+1}-t_{k}\right) + \eta b^{ k},$ and $\Lambda_{1}\left(t_{k+1}-t_{k}\right)>0$ we obtain 
\begin{equation*}
\left\|\varrho\left(t_{k+1}\right)\right\|^{2} \leq e^{-2\eta b^{ k}}\left\|\Psi\left(t_{k}\right)\right\|^{2}.
\end{equation*}
We have $\Psi\left(\cdot\right) = \Upsilon\left(\cdot\right)  + \varrho\left(\cdot\right)$. This implies that
\begin{equation}
\left\|\Psi\left(t_{k+1}\right)\right\|^{2} \leq\left\|\Upsilon\left(t_{k+1}\right)\right\|^{2}+\left\|\varrho\left(t_{k+1}\right)\right\|^{2} \leq \mathrm{e}^{1- \eta b^{ k}}\left\|\Psi\left(t_{k}\right)\right\|^{2}.
\end{equation}
By induction for any $k \geq 1$,
\begin{equation}\label{equ46}
    \left\|\Psi\left(t_{k}\right)\right\|^{2} \leq \mathrm{e}^{k- \eta b^{ k}}\left\|\Psi\left(t_{0}\right)\right\|^{2},
\end{equation}
Next, we estimate the control function $\mathcal{L}_{k}$ explicitly given in \eqref{eq35} associated with $\Psi$.
\begin{equation}\label{eq45}
\begin{aligned}
\left\|\mathcal{L}_{k}\left(\psi\left(t_{k}\right)\right)\right\|_{\omega}^{2} &=\left\| \sum_{\lambda_{i}\leq \Lambda_{k}}  a_{i} h_{i}\right\|_{\omega}^{2} \leq \int_{\omega}\left(\sum_{\lambda_{i} \leq \Lambda_{k}}\mid a_{i}\mid \mid h_{i} \mid\right)^{2} \\
&\leq \sum_{\lambda_{i} \leq \Lambda_{k}}\mid a_{i}\mid^{2} \sum_{\lambda_{i} \leq \Lambda_{k}}\left\|h_{i}\right\|_{\omega}^{2}\\
& \leq\left\|\Psi\left(t_{k}\right)\right\|^{2} \sum_{\lambda_{i} \leq \Lambda_{k}}\mathrm{e}^{C_{3}\left(1+\frac{2}{t_{k+1}-t_{k}}\right)}\mathrm{e}^{\frac{\sqrt{2}C_{3}}{\sqrt{t_{k+1}-t_{k}}}\sqrt{\ln \left( \mathrm{e}+ \mathrm{e}^{\eta b^{k}}\sum\limits_{\lambda_{i}\leq \Lambda_{k}}1 \right)}}\\
& \leq\left\|\Psi\left(t_{k}\right)\right\|^{2} \sum_{\lambda_{i} \leq \Lambda_{k}}\mathrm{e}^{C_{3}\left(1+\frac{2}{t_{k+1}-t_{k}}\right)}\mathrm{e}^{\frac{2 C_{3}}{\sqrt{t_{k+1}-t_{k}}}\sqrt{\ln \left( \mathrm{e}^{\eta b^{k}}\sum\limits_{\lambda_{i}\leq \Lambda_{k}}1 \right)}}
\end{aligned}
\end{equation}
In the third line we used \eqref{equa30}. On the other hand, by Young's inequality we obtain
\begin{equation}\label{eq46}
    \frac{2 C_{3}}{\sqrt{t_{k+1}-t_{k}}}\sqrt{\ln \left( \mathrm{e}^{\eta b^{k}}\sum\limits_{\lambda_{i}\leq \Lambda_{k}}1 \right)} \leq \frac{2 C_{3}^{2}}{t_{k+1}-t_{k}} + \frac{1}{2} \ln \left(\mathrm{e}^{\eta b^{k}}\sum_{\lambda_{i}\leq \Lambda_{k}}1 \right).
\end{equation}
Thus,
\begin{equation}\label{eq48}
\begin{aligned}
\left\|\mathcal{L}_{k}\left(\psi\left(t_{k}\right)\right)\right\|_{\omega}^{2}& \leq\left\|\Psi\left(t_{k}\right)\right\|^{2} \mathrm{e}^{C_{3}\left(1+\frac{2}{t_{k+1}-t_{k}}\right)}\mathrm{e}^{ \frac{2 C_{3}^{2}}{t_{k+1}-t_{k}} } \mathrm{e}^{ \frac{1}{2} \ln \left(\mathrm{e}^{\eta b^{k}}\sum_{\lambda_{i}\leq \Lambda_{k}}1 \right)} \sum_{\lambda_{i} \leq \Lambda_{k}} 1\\
&\leq\left\|\Psi\left(t_{k}\right)\right\|^{2} \mathrm{e}^{C_{3}\left(1+\frac{2}{t_{k+1}-t_{k}}\right)}\mathrm{e}^{ \frac{2 C_{3}^{2}}{t_{k+1}-t_{k}} } \mathrm{e}^{\frac{1}{2}\eta b^{k}} \left(\sum_{\lambda_{i} \leq \Lambda_{k}} 1\right)^{\frac{3}{2}}.
\end{aligned}
\end{equation}
using \eqref{equa34}, \eqref{equ34}, \eqref{equ46} and \eqref{eq48}, we obtain
\begin{equation}\label{eq51}
\begin{aligned}
\left\|\mathcal{L}_{k}\left(\psi\left(t_{k}\right)\right)\right\|_{\omega}^{2}& \leq \mathrm{e}^{k- \eta b^{ k}}\left\|\Psi\left(t_{0}\right)\right\|^{2} \mathrm{e}^{C_{3}\left(1+\frac{2}{T} \frac{b^{k+1}}{b-1}\right)}\mathrm{e}^{ \frac{2 C_{3}^{2}}{T} \frac{b^{k+1}}{b-1} } \mathrm{e}^{\frac{1}{2}\eta b^{k}}\left( C \Lambda_{k}^{\frac{n}{2}}\right)^{\frac{3}{2}}\\
&\leq \mathrm{e}^{k- \frac{1}{2}\eta b^{ k}}\left\|\Psi\left(t_{0}\right)\right\|^{2} \mathrm{e}^{C_{3}}\mathrm{e}^{\left(\frac{2}{T} \frac{b^{k+1}}{b-1}\right)\left(C_{3}+C_{3}^{2}\right) } \left( C\left(\lambda_{1}+\dfrac{\eta}{T}\dfrac{b^{2k+1}}{b-1} \right)^{\frac{n}{2}}\right)^{\frac{3}{2}}.
\end{aligned}
\end{equation}
Next, we choose $\eta >1$ as follows
\begin{equation*}
    \eta = 1+ 4 \left(\frac{2}{T} \frac{b}{b-1} \left(C_{3}+C_{3}^{2} \right)\right),
\end{equation*}
which implies that
\begin{equation*}
    -\frac{1}{2} \eta b^{ k}+\left(C_{3}+ C_{3}^{2}\right)\left(\frac{2}{T} \frac{b^{k+1}}{b-1}\right) \leq-\frac{1}{4} \eta b^{ k},
\end{equation*}
Since $b>1$, we obtain
\begin{equation*}
\begin{aligned}
\left\|\mathcal{L}_{k}\left(\psi\left(t_{k}\right)\right)\right\|_{\omega}^{2}\leq \mathrm{e}^{k- \frac{1}{4}\eta b^{\beta k}} \mathrm{e}^{C_{3}} \left( C\left(\lambda_{1}+\dfrac{\eta}{T}\dfrac{b}{b-1} \right)^{\frac{n}{2}}\right)^{\frac{3}{2}} b^{\frac{3nk}{2}} \left\|\Psi\left(t_{0}\right)\right\|^{2},
\end{aligned}
\end{equation*}
using the fact that 
\begin{equation*}
    b^{\frac{3 n k}{2}} \leq  \left( \frac{12n}{\eta}\right)^{\frac{3n}{2}} \mathrm{e}^{\frac{1}{8}\eta b^{k} },
\end{equation*}
we obtain that for any $k\geq 1$,
\begin{equation*}
\left\|\mathcal{L}_{k}\left(\psi\left(t_{k}\right)\right)\right\|_{\omega}^{2}\leq C_{4}\mathrm{e}^{k- \frac{1}{8}\eta b^{\beta k}} \left\|\Psi\left(t_{0}\right)\right\|^{2},
\end{equation*}
such that $C_{4}:=  \mathrm{e}^{C_{3}} \left( \frac{12n}{\eta}\right)^{\frac{3n}{2}} \left( C\left(\lambda_{1}+\dfrac{\eta}{T}\dfrac{b}{b-1} \right)^{\frac{n}{2}}\right)^{\frac{3}{2}}.$

For all $t\geq 0$, there exist $k>1$ such that $t \in [t_{k}, t_{k+1}]$. To reach our final result, we distinguish four cases:

\noindent If $t \in [t_{0},\tau_{0})$, then
\begin{equation*}
     \left\|\Psi\left(t\right)\right\|^{2} \leq \left\|\Psi\left(t_{0}\right)\right\|^{2};
\end{equation*}
If $t \in [\tau_{0}, t_{1})$, then
\begin{equation}
\begin{aligned}
 \left\|\Psi\left(t\right)\right\|^{2} &\leq \left\|\Psi\left(\tau_{0}\right)\right\|^{2}= \left\|\Psi\left(\tau_{0}^{-}\right) +\mathds{1}_{\omega} \mathcal{L}_{0}(\psi(\cdot,t_{0})) \right\|^{2}\\
 &\leq 2 \left\|\Psi\left(\tau_{0}^{-}\right)\right\|^{2} + 2 \left\| \mathcal{L}_{0}(\psi(\cdot,t_{0})) \right\|^{2}\\
 &\leq \left(2 + 2\left\| \mathcal{L}_{0} \right\|^{2} \right) \left\|\Psi\left(t_{0}\right)\right\|^{2};
\end{aligned}
\end{equation}
If $k \geq 1$ and $t \in [t_{k},\tau_{k})$, then
\begin{equation*}
     \left\|\Psi\left(t\right)\right\|^{2} \leq \left\|\Psi\left(t_{k}\right)\right\|^{2}\leq \mathrm{e}^{k- \eta b^{ k}}\left\|\Psi\left(t_{0}\right)\right\|^{2}.
\end{equation*}
If $k \geq 1$ and $t \in [\tau_{k}, t_{k+1})$, then
\begin{equation}\label{equ54}
\begin{aligned}
 \left\|\Psi\left(t\right)\right\|^{2} &\leq \left\|\Psi\left(\tau_{k}\right)\right\|^{2}= \left\|\Psi\left(\tau_{k}^{-}\right) +\mathds{1}_{\omega} \mathcal{L}_{k}(\psi(\cdot,t_{k})) \right\|^{2}\\
 &\leq 2 \left\|\Psi\left(\tau_{k}^{-}\right)\right\|^{2} + 2 \left\| \mathcal{L}_{k}(\psi(\cdot,t_{k})) \right\|^{2}\\
 &\leq 2 \mathrm{e}^{k- \eta b^{ k}}\left\|\Psi\left(t_{0}\right)\right\|^{2} + 2 C_{4}\mathrm{e}^{k- \frac{1}{8}\eta b^{ k}} \left\|\Psi\left(t_{0}\right)\right\|^{2}\\
  &\leq 2 \left(1+C_{4} \right)\mathrm{e}^{k- \frac{1}{8}\eta b^{ k}} \left\|\Psi\left(t_{0}\right)\right\|^{2}.
\end{aligned}
\end{equation} 
From \eqref{equ54} and by choosing $b = \mathrm{e}^{ \frac{16}{\eta}}$, we obtain
\begin{equation*}
\begin{aligned}
    \left\|\Psi\left(t\right)\right\|^{2} &\leq 2 \left(1+C_{4} \right)\mathrm{e}^{-\frac{1}{16}\eta b^{ k}} \left\|\Psi\left(t_{0}\right)\right\|^{2}\\
    &\leq 2 \left(1+C_{4}+\left\| \mathcal{L}_{0} \right\|^{2} \right)\mathrm{e}^{-\frac{1}{16}\eta b^{ k}} \left\|\Psi\left(t_{0}\right)\right\|^{2}.
\end{aligned}
\end{equation*}
Consequently, for any $t \in \left[t_{k},t_{k+1} \right]$
\begin{equation}
\begin{aligned}
    \left\|\Psi\left(t\right)\right\|^{2}
     \leq C \mathrm{e}^{-\frac{1}{16}\eta b^{ k}} \left\|\Psi\left(t_{0}\right)\right\|^{2}
\end{aligned}
\end{equation}
with $C:= 2 \left(1+C_{4}+\left\| \mathcal{L}_{0} \right\|^{2} \right) \mathrm{e}^{\frac{1}{16}\eta}$, and
\begin{equation*}
    b^{k} \leq \frac{T}{T-t} \leq b^{k+1}.
\end{equation*}
Then
\begin{equation*}
     \mathrm{e}^{-\frac{1}{16}\eta b^{ k}} \leq \mathrm{e}^{-\frac{1}{16}\frac{\eta T}{b(T-t)} } ,
\end{equation*}
which implies that
\begin{equation*}
\|\Psi(t)\| \leq C \mathrm{e}^{-\frac{1}{K}\left(\frac{T}{T-t}\right)} \left\|\Psi_{0}\right\| \text { for any } 0 \leq t<T,
\end{equation*}
with $K := \frac{16 b}{\eta} $. 
This provides the desired inequality.

\section{Conclusions and possible extensions}
In this work, we have dealt with the impulse null controllability of the heat equation with dynamic boundary conditions. First, we have generalized the main result in \citep{CGMZ'22} consisting on proving a logarithmic convexity estimate in a bounded convex domain. Here we have considered a more general case of $C^2$-domains by proving some local estimates. Then an explicit estimate of the exponential decay of the solution via impulse controls has been established, and the impulse null controllability has been obtained as a direct consequence.

To the best of the authors knowledge, there is no work in the literature dealing with the numerical computation of the impulse null controls for the heat equation even with static boundary conditions. It would be of much interest to investigate this problem. Furthermore, it should be noted that we have proved the approximate impulsive null controllability of the same system using only one pulse arbitrarily located in the time interval \citep{CGMZ'22}, in this work we have been obliged to consider a sequence of well-chosen pulses over the time horizon to establish the impulse null controllability. This raises the question whether it is possible to reach the null controllability from a single pulse. Such a result will generalize several works in the literature that deal with the null controllability of the heat equation by a control acting on the whole time interval.
\section*{Data availability statement}
The authors state that no data sharing is applicable to this article because no data sets were generated or analyzed during the current study.

\end{document}